\documentclass[hidelinks,onefignum,onetabnum]{siamart250211}

\usepackage{stmaryrd,MnSymbol}
\usepackage{graphicx}
\usepackage{subcaption}
\usepackage{caption}
\usepackage{framed}
\usepackage{booktabs}
\usepackage[color=yellow]{todonotes}

\newcommand{\pp}[2]{\frac{\partial #1}{\partial #2}}

\DeclareMathOperator{\diff}{d}

\newtheorem{remark}[theorem]{Remark}

\usepackage{color}

\DeclareMathOperator{\mxd}{mixed}
\newcommand{\Smixed}{S_{\mxd}}
\newcommand{\Hmixed}{H_{\mxd}}

\DeclareMathOperator{\prm}{primal}
\newcommand{\Sprimal}{S_{\prm}}
\newcommand{\Hprimal}{H_{\prm}}

\usepackage{mathtools}
\DeclarePairedDelimiter\ceil{\lceil}{\rceil}

\usepackage{fancybox}
\begin{document}
\title{Shifted HSS solvers for the indefinite Helmholtz
  equation} \author{C. J. Cotter, J. Hope-Collins and K. Knook}
\maketitle

\begin{abstract}
  We provide an iterative solution approach for the indefinite
  Helmholtz equation discretised using finite elements, based upon a
  Hermitian Skew-Hermitian Splitting (HSS) iteration applied to the
  shifted operator, and prove that the iteration is $k$- and
  mesh-robust when $\mathcal{O}(k)$ HSS iterations are performed.  The
  HSS iterations involve solving a shifted operator that is suitable
  for approximation by multigrid using standard smoothers and transfer
  operators, and hence we can use $\mathcal{O}(N)$ parallel processors
  in a high performance computing implementation, where $N$ is the
  total number of degrees of freedom.  We argue that the algorithm
  converges in $\mathcal{O}(k)$ wallclock time when within the range
  of scalability of the multigrid. We provide numerical results in
  both 2D and 3D verifying our proofs and demonstrating this claim,
  establishing a method that can make use of large scale high
  performance computing systems.
\end{abstract}

% REQUIRED

% REQUIRED

\section{Introduction}

In this article, we consider the indefinite Helmholtz
equation\footnote{This formulation deviates slightly from the usual
parameters $-k^2-i\epsilon$; when $\delta \ll k$ we have $(-\delta +
ik)^2 \approx -k^2 -2i\delta k$, so $\epsilon\approx 2\delta k$; this
choice is made to simplify formulae later in a mixed formulation.},
\begin{equation}
  \label{eq:helmholtz strong primal}
  (-\delta + ik)^2u - \nabla^2 u = f,
\end{equation}
solved in the domain $\Omega$ with boundary $\partial\Omega$,
with boundary conditions $u=0$ on $\partial\Omega_0\subset \partial\Omega$ and
$\pp{u}{n}-(-\delta+ik)u=0$ on $\Gamma \subset \partial\Omega$, with
$\partial\Omega=\partial\Omega_0+\Gamma$, and real parameters
$\delta\geq 0$ and $k>0$. This
equation arises from the time Fourier transformation of the wave
equation in time with coefficient $k$ and absorption parameter
$\delta$, with Sommerfeld radiation boundary conditions on $\Gamma$;
it serves as a prototype for more complicated wave systems such as
those arising in electromagnetism, elastodynamics, etc.

The design of efficient and scalable iterative solvers for
\eqref{eq:helmholtz strong primal} continues to be a challenging area
of scientific computing when $|k|\gg \delta$, as surveyed by
\cite{ernst2011difficult}. A number of sophisticated preconditioners
have been developed. These include the analytic preconditioners which
use Pad\'e approximations of integral operators \cite[and citations
  therein]{darbas2013combining}; these are mainly suited to
homogeneous media.  There is a large literature on domain
decomposition methods \cite[for
  example]{galkowski2025convergence,graham2017domain,gander2007optimized,vion2014double}
that combine the solution on (possibly overlapping) patches with
various boundary conditions. Sweeping preconditioners
\cite{engquist2011sweeping,engquist2011a_sweeping} cut the domain into
subdomains, using Sommerfeld conditions or perfectly matched layers on
the interior boundaries to approximate the Schur complement
eliminating the solution away from those boundaries. By using a
checkerboard pattern to organise the subdomains, \cite{taus2020sweeps}
achieved a parallel wallclock time complexity of $\mathcal{O}((N/p)\log k)$ where
$N=n^d$ is the number of degrees of freedom in \(d\) spatial dimensions and $p=\mathcal{O}(n)$ is
the total number of processors. Time domain preconditioning uses
timestepping methods for the time domain wave equation to construct
preconditioners for the frequency domain Helmholtz equation
\cite{stolk2021time}.

In this article, we consider the solution of the indefinite Helmholtz
problem using multigrid methods with classical smoothers (such as
scaled Jacobi iteration) and classical grid transfer operators, in
combination with shift preconditioning. The challenge is that on the
one hand, classical multigrid methods with standard smoothers and
transfer operators only work well when $\delta = \mathcal{O}(k)$ as
$k\to\infty$ \cite{cocquet2017large}. On the other hand, we typically
want to solve problems with $\delta=\mathcal{O}(1)$ ($\delta=0$, in
particular). Recently, \cite{dwarka2023stand} presented a multigrid
approach using standard smoothers but modified transfer operators
adapted to the large $k$ situation. Here, we pursue an alternative
direction: the Helmholtz operator can be preconditioned with a shifted
Helmholtz operator with $\delta$ replaced by
$\hat{\delta}>\delta$. The inverse of the shifted operator can then be
further approximated with one or more multigrid cycles.  If we choose
$\hat{\delta}=\mathcal{O}(k)$, then the shifted operator will be
suitable for multigrid, i.e. standard multigrid will give $k$-robust
convergence.  However, \cite{gander2015applying} showed that
$\hat{\delta}$ should be $\mathcal{O}(1)$ as $k\to\infty$ for
$k$-robust preconditioning of the original operator, i.e. in order
that the number of required GMRES iterations $n$ is independent of $k$
as $k\to \infty$.  This $k$-robustness is important because the cost
of GMRES is $\mathcal{O}(n^2)$, and restarting is not generally
robust.

In this article we describe how to use Hermitian Skew-Hermitian
Splitting (HSS) iteration \cite{bai2003hermitian} to bridge the gap
between these requirements {$\hat{\delta}=\mathcal{O}(k)$} and
$\hat{\delta}=\mathcal{O}(1)$, focussing on the case of finite element
discretisations (the basic idea can be easily adapted to other
methods). HSS is a two step stationary iterative method based on
splitting the operator into Hermitian and skew-Hermitian parts. After
eliminating down to one step, the skew-Hermitian part only appears in
its Cayley transform, which is a unitary operator that can be removed
from the upper bound for the contraction rate. There are a number of
extensions of HSS iteration in the literature, such as having
different coefficients for the two steps, replacing the Hermitian part
with other structures such as upper triangular matrices, adapting the
method to saddle point problems, and acceleration techniques \cite[for
  example]{bai2010modified,bai2007accelerated,wang2004skew,salkuyeh2015generalized,li2014lopsided}. \cite{li2024two}
considered an HSS-like splitting for the Helmholtz equation, obtaining
systems that are amenable to a discrete sine transformation (rather
than multigrid).

In this article, when $\delta > 0$ we formulate a PHSS (Preconditioned
Hermitian Skew-Hermitian Splitting) iteration
\cite{bertaccini2005preconditioned} that involves solving Helmholtz
problems with $\hat{\delta}=\mathcal{O}(k)$, which is sufficient for
robust approximate solution using multigrid. It is then
straightforward to show that $k$ HSS iterations produces an
$\mathcal{O}(1)$ reduction in the error, independent of mesh
resolution.  For $\delta=0$, we propose to precondition the indefinite
Helmholtz operator with the $\delta$-shifted operator with
$\hat{\delta}=\mathcal{O}(1)$, which in turn is approximated by some
chosen number of HSS ``inner'' iterations.  We show that if the number
of inner iterations is again chosen to be $\mathcal{O}(k)$, then the
outer iteration for the unshifted operator is $k$- and
mesh-robust. This proof relies on results from
\cite{gander2015applying} which use a different norm than the natural
one for the HSS computations, and some additional technicalities arise
from that. Although we focus on the $\delta=0$ problem, the method and
theoretical results are easily extendable to the case of small
positive $\delta$ (or spatially varying $\delta$ that becomes zero or
small in some regions).

When the HSS iterations for the $\hat{\delta}=\mathcal{O}(k)$ system
are approximated using multigrid, the mesh-robustness is maintained,
and the algorithm can be parallelised using standard Message Passing
Interface (MPI) parallelisation techniques implemented in standard
libraries such as PETSc \cite{petsc-user-ref}. Since the whole
iteration is built from standard components, both the smoothers and
the transfer operators have fixed stencil widths independent of $k$,
so in turn the halos in the MPI communications have fixed depth, and
we expect to see very good parallel scaling. In the range of perfect
weak scaling of the multigrid solves, given sufficient computational
cores, we argue that the algorithm should converge in $\mathcal{O}(k)$
wallclock time.  This is beneficial when the mesh size $h$ needs to go
to zero faster than $1/k$ as is required to overcome the pollution
effect \cite{BaylissA1985Oacf}. \cite{spence2025preconditioning}
proved that for the interior impedance problem, convergence of
numerical solutions is obtained when $hk^{3/2}$ is sufficiently small
(with $(hk)^{2p}k$ sufficiently small for degree $p$ finite elements
and suitably smooth boundary).

The rest of the article is organised as follows. In Section \ref{sec:formulations}
we establish notation for two formulations of the indefinite Helmholtz equation
and their finite element discretisation. In Section \ref{sec:HSS shifted},
we introduce the HSS iteration for the shifted ($\delta > 0$) Helmholtz equation,
proving $k$- and mesh-robustness. In Section \ref{sec:HSS 0}, we introduce
the shifted HSS iteration for the unshifted ($\delta=0$) Helmholtz equation,
proving $k$- and mesh-robustness. In Section \ref{sec:numerics} we provide
numerical results that verify these proofs, and demonstrate our claim of
$\mathcal{O}(k)$ wallclock solver time. Finally, in Section \ref{sec:summary}
we provide a summary and outlook.

\section{Primal and mixed formulations}
\label{sec:formulations}

In this section we briefly introduce the variational formulations that
we will address.  We consider two
formulations that lead to closely related iterative
strategies.

Starting from \eqref{eq:helmholtz strong primal}, the
``primal'' formulation is obtained by multiplying by a test function
and integrating by parts, and restricting solution and test function
to the subspace
\begin{equation}
  H^1_0(\Omega) = \{v\in H^1(\Omega):v=0 \mbox{ on }\partial\Omega_0\}.
\end{equation}
The primal formulation thus seeks $u\in H^1_0(\Omega)$, such that
\begin{equation}
  \label{eq:primal H1}
  b_\delta(u,v)
  = G_u[v], \quad \forall v\in H^1_0(\Omega),
\end{equation}
where
\begin{align}
  b_\delta(u,v) &=\int_\Omega (-\delta+ik)^2uv^* + \nabla u\cdot\nabla v^* \diff x
  - (-\delta+ik)\int_{\Gamma} uv^*\diff S, \\
  G_u[v] &= \int_\Omega fv^* \diff x,
\end{align}
and where $^*$ is the complex conjugate.
To form the finite element discretisation we can choose a finite
element space $Q_h\subset H^1_0(\Omega)$, i.e. a continuous finite
element space. To be concrete, we choose a triangular/tetrahedral
mesh with continuous piecewise polynomials of degree $p$, which we
denote $\text{CG}_p$, and select the subspace of functions that satisfy the
boundary condition on $\partial\Omega_0$. Then the finite element
discretisation becomes: find $u\in Q_h$, such that
\begin{equation}
  \label{eq:primal}
  b_\delta(u,v) = G[v], \quad \forall v\in Q_h.
\end{equation}

The ``mixed'' formulation we consider here proceeds from the first order system
\begin{equation}
  (\delta-ik)\sigma - \nabla u = 0,
  \quad -\nabla\cdot\sigma + (\delta-ik)u = \frac{f}{\delta-ik},
\end{equation}
with equivalent boundary conditions $u=0$ on $\partial\Omega_0$ and
$\sigma\cdot n = -u$ on $\Gamma$ (because
$\pp{u}{n}=-(\delta+ik)\sigma$). Our corresponding variational
formulation seeks $(\sigma,u)\in (L^2(\Omega))^d\times
H^1_0(\Omega)$ such that
\begin{align}
  a_\delta((v, \tau), (u,\sigma)) = F_\sigma[\tau] + F_u[v],
  \quad \forall (v,\tau) \in (L^2(\Omega))^d\times
H^1_0(\Omega),
\end{align}
where 
\begin{align} \nonumber
  a_\delta((v, \tau), (u,\sigma)) =& \int_\Omega (\delta-ik)\sigma\cdot \tau^* - \nabla u\cdot\tau^* \diff x \\
  & \quad + \int_\Omega \sigma\cdot \nabla v^* + (\delta-ik)uv^* \diff x + \int_\Gamma u v^* \diff S, \\
 F_u[v] =&  \int_\Omega \frac{fv^*}{\delta - ik} \diff x,\\
 F_\sigma[\tau]  =& 0,
 \end{align}
(Note that we could have
also chosen a formulation $(\sigma, u)\in H(\mbox{div};\Omega)\times
L^2(\Omega)$, but we leave this choice for future consideration.) To
define the mixed finite element discretisation, we need to select
$V_h\times Q_h\subset (L^2(\Omega))^d\times H^1_0(\Omega)$
appropriately. Here we choose $Q_h$ to be the $\text{CG}_p$ space as before
(selecting the subspace that satisfies the boundary conditions), and
we choose $V_h$ to be the vector function space with piecewise polynomials of degree $p-1$ for each entry, denoted by $(\text{DG}_{p-1})^d$.  This
choice gives the embedding property $u \in
Q_h\implies \nabla u\in V_h$. The mixed finite element discretisation
then seeks $(\sigma,u)\in V_h\times Q_h$ such that
\begin{equation}
  \label{eq:mixed}
    a_\delta((v, \tau),(u,\sigma)) = F_\sigma[\tau] + F_u[v],
  \quad \forall (v,\tau)\in V_h\times Q_h.
\end{equation}
The conditions for well posedness and quasi-optimality of \eqref{eq:primal H1} and the finite
element discretisation \eqref{eq:primal} are well established
\cite{melenk1995generalized,cummings2006sharp,hetmaniuk2007stability, SchatzAlfredH.1974AOCR},
and here we just assume throughout the article that they hold. To
establish well-posedness for the discrete mixed formulation
\eqref{eq:mixed}, we just note that the embedding property
means that we can choose $\tau = (\delta-ik)\sigma - \nabla u$,
meaning that $(\delta-ik)\sigma - \nabla u=0$ in $L^2$. Then we
can eliminate $\sigma$ to recover the primal formulation \eqref{eq:primal}, and
well posedness and quasi-optimality follow.

\section{HSS iteration for the shifted system}

\label{sec:HSS shifted}

In this section we define stationary iterative methods for the
shifted (\emph{i.e.} $\delta>0$) primal and mixed systems. The
stationary method can be used as a preconditioner, or a solver in its
own right.  We motivate this by formulating an HSS iteration using
operators on the finite element space, before reformulating as a
standard variational formulation that demonstrates how to implement
the iteration in practice, before analysing the convergence.

\subsection{Primal formulation}
First, we consider the primal shifted problem \eqref{eq:primal}.
Multiplication by $i$ gives
\begin{equation}
  \int_\Omega (i(\delta^2 - k^2)+2\delta k)uv^* + i\nabla u\cdot\nabla v^* \diff x
  + (\delta i+k)\int_{\Gamma} uv^*\diff S
  = i\int_\Omega fv^* \diff x, \quad \forall v \in V_h.
\end{equation}
Then, we define the skew-Hermitian operator $\Sprimal $ such that
\begin{equation}
  \langle v, \Sprimal u \rangle_{\Hprimal} = i(\delta^2 - k^2)\langle v, u\rangle
  + i\langle \nabla v, \nabla u\rangle
  + \delta i \llangle v, u\rrangle, \quad \forall u,v \in V_h,
\end{equation}
where we introduced the inner product $\langle \cdot,\cdot\rangle_{\Hprimal}$,
defined by 
\begin{equation}
  \langle v, u \rangle_{\Hprimal} = 2\delta k\langle v, u\rangle
  + k\llangle v, u \rrangle,
\end{equation}
where $\langle\cdot,\cdot\rangle$ denotes the usual $L^2(\Omega)$ inner product
for scalar complex valued functions defined  (with similar definition for vector fields) as
\begin{equation}
  \langle u, v \rangle_\Omega = \int_\Omega uv^*\diff x,
\end{equation}
and $\llangle \cdot,\cdot\rrangle$ denotes the $L^2(\partial\Omega)$
inner product for functions restricted to the boundary
$\partial\Omega$ of $\Omega$, defined similarly.

Our problem becomes
\begin{equation}
  (I + \Sprimal )u = \hat{f},
\end{equation}
where
\begin{equation}
  \langle v, \hat{f} \rangle_{\Hprimal} = i\langle v, f\rangle.
\end{equation}
Here, with respect to the $\Hprimal$ inner product, $I$ is Hermitian
and $\Sprimal$ is skew-Hermitian. Thus, we can write a two-step
stationary iterative method based on the Hermitian/skew-Hermitian
splitting (HSS) for this system as
\begin{equation}
  (\gamma + 1)u_{n+1/2} = (\gamma I - \Sprimal)u_n + f, \quad
  (\gamma I + \Sprimal)u_{n+1} = (\gamma - 1)u_{n+1/2} + f.
\end{equation}
with $\gamma > 0$. Eliminating $u_{n+1/2}$ gives
\begin{equation}
  \label{eq:HSS}
  (\gamma I + \Sprimal)u_{n+1} = \frac{\gamma - 1}{\gamma + 1}
  (\gamma I - \Sprimal)u_n + \frac{2\gamma}{\gamma + 1}f.
\end{equation}
The introduction of the ${\Hprimal}$ inner product is equivalent to
formulating a preconditioned HSS method
\cite{bertaccini2005preconditioned} using the Hermitian part of the
matrix as the preconditioner.

The optimal value of $\gamma$ for HSS iteration with the Hermitian
part being the identity is $\gamma=1$, which converges in 1 iteration,
but this is just equivalent to solving the original problem
\eqref{eq:primal}, and we have not gained anything. Instead,
in this article we choose $\gamma = k$. In each iteration, we have to
solve the shifted operator $kI + \Sprimal$ which is more suited to
multigrid than $I+\Sprimal$; this is more clearly seen by inspecting
the equivalent variational formulation below.
\begin{definition}[$\gamma=k$ HSS iteration in primal form]
  The $\gamma=k$ HSS iteration for the system \eqref{eq:primal}
  produces a sequence $u_0,u_1,\ldots$, defined
  iteratively for $n\geq0$ by finding $u_{n+1}\in V_h$ such that
  \label{def:primal hss}
  \begin{align}\label{eq:primal-hss-system}
    & \left(- 2\delta k^2i + \delta^2 - k^2 \right)\langle v, u_{n+1} \rangle
    + (-k^2i + \delta)\llangle v , u_{n+1} \rrangle
    + \langle \nabla v, \nabla u_{n+1}\rangle \nonumber
    \\
    & \quad = \frac{k-1}{k+1}\left((-2\delta k^2i - \delta^2 + k^2)\langle v, u_{n} \rangle
    + (-k^2 i - \delta)\llangle v , u_{n} \rrangle
    - \langle \nabla v, \nabla u_{n}\rangle\right) \nonumber \\
    & \quad + \frac{2k}{k+1}\langle v, f \rangle, \quad
    \forall v.
  \end{align}
\end{definition}
(Note that we have divided by $i$ again in
\eqref{eq:primal-hss-system}.)  Comparing \eqref{eq:primal} with
\eqref{eq:primal-hss-system}, we see that $(-\delta +
ik)^2=-k^2+\delta^2 -2i\delta k$ has been replaced by $-k^2-\delta^2
- 2i\delta k^2$. Considering $\delta$ fixed as $k$ gets large, this
shift is the minimum necessary scaling with $k$ for efficient
inversion by multigrid as discussed in \cite{cocquet2017large}. The
only difference in our present formulation is that the boundary term
is scaled differently. Presently, we have not analysed multigrid
applied to this modification but our numerical experiments demonstrate
that this shifting is sufficient to obtain $k$ independent multigrid
convergence.

%% \begin{align}
%%   \label{eq:HSS sigma} \nonumber
%%   k \langle \tau, (\delta - i)\sigma_{n+1} \rangle
%%   - \langle \tau, \nabla u_{n+1} \rangle
%%    & = & \frac{k-1}{k+1}\left(\langle \tau, k(\delta + i)\sigma_n \rangle
%%     + \langle \tau, \nabla u_n\rangle\right) \\
%%   & & \quad + \frac{2k}{k+1}F_\sigma[\tau], 
%%   \quad \forall \tau \in V_h, \\
%%     \label{eq:HSS u} \nonumber
%%   k\langle v, (\delta - i)u_{n+1} \rangle
%%   + \langle \nabla v, \sigma_{n+1} \rangle & \\
%%   \quad + {k}\llangle v, u_{n+1} \rrangle
%%   & = & \frac{k-1}{k+1}\left(k\langle v, (\delta + i)u_n \rangle
%%   - \langle \nabla v, \sigma_{n} \rangle
%%   + {k}\llangle v, u_n \rrangle
%%   \right) \nonumber \\
%% & & \qquad\qquad  + \frac{2k}{k+1}F_u[v],
%%   \quad \forall v \in Q_h.
%% \end{align}
%% \end{definition}

Having understood the benefits of the choice $\gamma=k$, we can
analyse the convergence of our preconditioned HSS scheme \eqref{eq:primal-hss-system}.
In a scalable implementation, the iteration is solved
approximately using a multigrid preconditioned Krylov method. However,
here we concentrate on the case where the iteration is solved exactly.

\begin{proposition}[HSS convergence for the shifted primal system]
  \label{prop:hss primal}
  The primal preconditioned HSS scheme \eqref{eq:primal-hss-system} has the error contraction rate
  bound
  \begin{equation}\label{eq:primal-hss-contraction}
    \|e_{n+1}\|_{\Hprimal} \leq \frac{1-1/k}{1+1/k}\|e_n\|_{\Hprimal},
  \end{equation}
    where
  \begin{equation}
    \|e\|_{\Hprimal}^2 = \langle e, e \rangle_{\Hprimal}.
  \end{equation}
\end{proposition}
\begin{proof}
  (Follows standard HSS arguments.)
  Writing $e_n=u_n-u_*$, where $u_*$ is the exact solution,
  the error equation is
  \begin{equation}
    e_{n+1} = \frac{k-1}{k+1}(kI+\Sprimal)^{-1}(kI-\Sprimal)e_n.
  \end{equation}
  Then,
  \begin{equation}
    \|e_{n+1}\|_{\Hprimal} = \frac{1-1/k}{1+1/k}\|(I + \Sprimal/k)^{-1}(I - \Sprimal/k)e_n\|_{\Hprimal}
    \leq \frac{1-1/k}{1+1/k}\|e_n\|_{\Hprimal},
  \end{equation}
  since $(I+ \Sprimal/k)^{-1}(I- \Sprimal/k)$ is the Cayley
  transform of $\Sprimal/k$, which is skew Hermitian in the ${\Hprimal}$ norm, and
  hence is ${\Hprimal}$ norm preserving.
\end{proof}
\begin{corollary}
  \label{eq:hss est primal}
  If we take $m\ceil{k}$ iterations of the HSS iteration for $m>0$, then the
  error reduction satisfies
  \begin{equation}
    \|e_{m\ceil{k}}\|_{\Hprimal} \leq c\|e_0\|_{\Hprimal},
  \end{equation}
  for $0<c<1$ independent of $k$ for the primal formulation.
\end{corollary}
\begin{proof}
    The previous result gives
  \begin{equation}
    \|e_{m\ceil{k}}\|_{\Hprimal} \leq \left(\frac{1-1/k}{1+1/k}\right)^{m\ceil{k}}\|e_0\|_{\Hprimal}.
  \end{equation}
  Using the Bernoulli estimate $(1+x) \leq \exp(x)$, we have
  \begin{equation}
    0 < \left(\frac{1-1/k}{1+1/k}\right)^{m\ceil{k}}
    \leq     \left(1-\frac{1}{k}\right)^{m\ceil{k}}
    \leq     \left(1-\frac{1}{\ceil{k}}\right)^{m\ceil{k}}
    \leq \exp(-m) < 1, \\
  \end{equation}
  so we have a contraction rate $\exp(-m)$
  which is independent of $k$.
\end{proof}
We conclude that the HSS iteration will converge to a chosen tolerance
in the $\Hprimal$ norm in $\mathcal{O}(k)$ iterations.

As discussed above, the whole point of this framework is to replace the
exact solution of \eqref{eq:primal-hss-system} with an approximation solution
using multigrid with standard components, perhaps as a preconditioner
for a Krylov method. This will be described in more detail in Section
\ref{sec:numerics}.

\subsection{Mixed formulation}
Now, we consider the shifted  mixed system \eqref{eq:mixed},
We define the operator $\Smixed:V_h\times Q_h\to V_h\times Q_h$,
\begin{equation}
  \langle \Smixed(\sigma, u), (\tau, v) \rangle_{\Hmixed}
  = -ik\langle \tau, \sigma \rangle -ik \langle v, u\rangle
  + \langle \nabla v, \sigma \rangle - \langle \tau, \nabla u \rangle,
\end{equation}
via the $\Hmixed$ inner product,
\begin{equation}
  \langle (\sigma, u), (\tau, v)\rangle_{\Hmixed}
  = \delta\langle \sigma, \tau\rangle
  + \delta\langle u, v\rangle
  + \llangle u, v\rrangle,
\end{equation}
We note that
$\Smixed$ is skew-Hermitian in the ${\Hmixed}$ inner product. We also
note that the $\Hmixed$ inner product is only an inner product for
$\delta > 0$, which explains why we need to consider a shifted system.

Thus, our shifted system \eqref{eq:primal} can be written as
\begin{equation}
  {(I + \Smixed)}{(\sigma, u)} = f,
\end{equation}
where $f$ is the Riesz representer of $F$ in the ${\Hmixed}$ inner
product.  Again, the operator has the form $I+\Smixed$ with
$\Smixed$ skew-Hermitian. Proceeding as for the primal case, we have a
stationary method,
\begin{equation}
  \label{eq:HSSm}
  (\gamma I + \Smixed)U_{n+1} = \frac{\gamma - 1}{\gamma + 1}
  (\gamma I - \Smixed)U_n + \frac{2\gamma}{\gamma + 1}f.
\end{equation}
This is expressed in variational formulation in the following definition.
\begin{definition}[$\gamma=k$ HSS iteration in mixed form]
  The $\gamma=k$ HSS iteration for the system \eqref{eq:mixed}
  produces a sequence $(\sigma_0,u_0),(\sigma_1,u_1),\ldots$, defined
  iteratively for $n>=0$ by finding $(\sigma_{n+1},u_{n+1})$
  in $V_h\times Q_h$ such that
  \label{def:mixed hss}
\begin{align}
  \label{eq:HSS sigma} \nonumber
  k \langle \tau, (\delta - i)\sigma_{n+1} \rangle
  - \langle \tau, \nabla u_{n+1} \rangle
   & = & \frac{k-1}{k+1}\left(\langle \tau, k(\delta + i)\sigma_n \rangle
    + \langle \tau, \nabla u_n\rangle\right) \\
  & & \quad + \frac{2k}{k+1}F_\sigma[\tau], 
  \quad \forall \tau \in V_h, \\
    \label{eq:HSS u} \nonumber
  k\langle v, (\delta - i)u_{n+1} \rangle
  + \langle \nabla v, \sigma_{n+1} \rangle & \\
  \quad + {k}\llangle v, u_{n+1} \rrangle
  & = & \frac{k-1}{k+1}\left(k\langle v, (\delta + i)u_n \rangle
  - \langle \nabla v, \sigma_{n} \rangle
  + {k}\llangle v, u_n \rrangle
  \right) \nonumber \\
& & \qquad\qquad  + \frac{2k}{k+1}F_u[v],
  \quad \forall v \in Q_h.
\end{align}
\end{definition}
To solve this system, we make use of the embedding property to
deduce from \eqref{eq:HSS sigma} that
\begin{equation}
  \label{eq:strong sigma}
  k(\delta - i)\sigma_{n+1}
  - \nabla u_{n+1} = \frac{k-1}{k+1}k(\delta + i)\sigma_n
  + \nabla u_n = \frac{2k}{k+1}\sigma_0,
\end{equation}
in the $L^2$ sense,
where
\begin{equation}
  \langle \tau, \sigma_0 \rangle = F_\sigma[\tau], \quad
  \forall \tau \in V_h.
\end{equation}
Using \eqref{eq:strong sigma} to eliminate $\sigma_{n+1}$
from  \eqref{eq:HSS u} leads to a system of the form
\begin{equation}
  k^2(\delta - i)^2\langle v, u_{n+1} \rangle
  + \langle \nabla v, \nabla u_{n+1} \rangle
  + k^2(\delta - i)\llangle v, u_{n+1} \rrangle
  = \tilde{F}_u[v], \forall v \in V_h,
\end{equation}
which is a primal shifted Helmholtz system  with the imaginary shift
$-2i\delta k$ replaced by $-2i\delta k^2$.

This looks almost identical to our primal HSS system after eliminating
$\sigma_{n+1}$, except without $\sigma_n$ on the RHS, and with a
slightly different boundary term. We briefly summarise corresponding
results for the mixed system.

\begin{proposition}[HSS convergence for the shifted mixed system]
  \label{prop:hss mixed}
  The mixed preconditioned HSS scheme (\ref{eq:HSS sigma}-\ref{eq:HSS u})
  has the error contraction rate
  bound
  \begin{equation}
    \|e_{n+1}\|_{\Hmixed} \leq \frac{1-1/k}{1+1/k}\|e_n\|_{\Hmixed},
  \end{equation}
  where
  \begin{equation}
    \|e\|_{\Hmixed}^2 = \langle e, e \rangle_{\Hmixed}.
  \end{equation}
\end{proposition}
\begin{proof}
  The proof is identical to the mixed case upon substitution of the
  definitions of $\langle\cdot,\cdot\rangle_{\Hmixed}$ and $\Smixed $
  for the primal case.
\end{proof}
\begin{corollary}
  \label{eq:hss est}
  If we take $m\ceil{k}$ iterations of the HSS iteration (\ref{eq:HSS
    sigma}-\ref{eq:HSS u}) for $m>0$, then the error reduction
  satisfies
  \begin{equation}
    \|e_{m\ceil{k}}\|_{\Hmixed} \leq c\|e_0\|_{\Hmixed},
  \end{equation}
  for $0<c<1$ independent of $k$ for the mixed formulation.
\end{corollary}
\begin{proof}
  The proof is identical to that of Corollary \ref{eq:hss est}.
\end{proof}
We conclude that the HSS iteration will converge to a chosen tolerance
in the $\Hmixed$ norm in $\mathcal{O}(k)$ iterations.

\section{Shifted HSS iteration for the $\delta=0$ system}
\label{sec:HSS 0}

In this section we describe and analyse our proposed shifted HSS
method for the $\delta=0$ system. We build this in two steps. 
First, following the shift preconditioning
approach, the system with $\delta=0$ is ``shift preconditioned'' by
the system with some chosen $\hat{\delta} > 0$. If the operator for the
$\delta$-shifted system is denoted {$B_\delta$ for the primal formulation and $A_\delta$ for the mixed formulation}, this leads to an iterative
scheme of the form,
\begin{equation}
  B_{\hat{\delta}} u^{n+1} = (B_{\hat{\delta}}-B_0)u^n + R,
\end{equation}
{for the primal formulation and similarly for the mixed formulation.}
As discussed in the introduction, $\hat{\delta}=\mathcal{O}(1)$ is
necessary for $k$ independent convergence rate bounds, whilst
$\hat{\delta}=\mathcal{O}(k)$ is necessary for $k$-robust solution of
the shifted system by multigrid with standard components.  To
reconcile this dichotomy, we replace the inverse of the shifted
operator ({$B_{\hat\delta}$ or} $A_{\hat{\delta}}$) by $m\ceil{k}$ ``inner'' iterations of
Definition \ref{def:primal hss} (or \ref{def:mixed hss} as
appropriate) with initialisation at zero. This is stated more
precisely in the following definitions.
\begin{definition}[Shifted HSS iteration for the primal system]
  \label{def:primal shifted HSS}
  For the primal system $B_0(u)[v]=G[v]$ and given some chosen integer
  $m>0$, we define the shifted HSS preconditioner
  $\tilde{B}_{\hat{\delta}}^{-1}:V_h^*\to V_h$ as $\tilde{B}_{\hat{\delta}}^{-1}R=
  u_{m\ceil{k}}$ where
  \begin{align}
    \label{eq:HSSpc primal}
    & \left(- 2\hat{\delta} k^2i + \hat{\delta}^2 - k^2 \right)\langle v, u_{n+1} \rangle
+ (-k^2i + \hat{\delta})\llangle v , u_{n+1} \rrangle
+ \langle \nabla v, \nabla u_{n+1}\rangle \nonumber
 \\
 = & \frac{k-1}{k+1}\left((-2\hat{\delta} k^2i - \hat{\delta}^2 + k^2)\langle v, u_{n} \rangle
+ (-k^2 i - \hat{\delta})\llangle v , u_{n} \rrangle
- \langle \nabla v, \nabla u_{n}\rangle\right) \nonumber \\
& 
+ \frac{2k}{k+1}R[v], \quad
\forall v, \quad n=0,\ldots m\ceil{k}-1,
  \end{align}
  and $u_0=0$.
\end{definition}
Hence, one application of $\tilde{B}^{-1}_{\hat{\delta}}$ involves $m\ceil{k}$
inner iterations. As usual, the preconditioner $\tilde{B}_{\hat{\delta}}$ can be
used as a preconditioner for a Krylov method to solve the system $B_0u=G$, or
to define a stationary iteration $\tilde{B}_{\hat{\delta}}u^{n+1} =
(\tilde{B}_{\hat{\delta}} - B_0)u^n + G$. The latter can equivalently
be thought of as a Richardson iteration preconditioned by $\tilde{B}_{\hat{\delta}}$.

\begin{definition}[Shifted HSS preconditioner for the mixed system]
  \label{def:shifted HSS preconditioning mixed}
  For the mixed system $A_0(\sigma, u)[\tau,v]=F[\tau,v]$ and given
  some chosen integer $m>0$, we define the
  shifted HSS preconditioner $\tilde{A}_{\hat{\delta}}^{-1}:V_h^*\times Q_h^*\to V_h\times Q_h$ as $\tilde{A}_{\hat{\delta}}^{-1}R=
  (\sigma_{m\ceil{k}},u_{m\ceil{k}})$ where
  \begin{align}
    \label{eq:HSSpc u}
    \langle \tau, k(\hat{\delta} - i) \sigma_{n+1}\rangle
    - \langle \tau, \nabla u_{n+1} \rangle
    & = & \frac{k-1}{k+1}\left( \langle \tau, k(\hat{\delta} + i)\sigma_n \rangle
    + \langle \tau, \nabla u_n \rangle\right)
    \nonumber \\
    & &
    +
    \frac{2k}{k+1}R_\sigma[\tau],
    \quad \forall \tau \in V_h,
    \\ \nonumber
    \langle v, k(\hat{\delta} - i) u_{n+1}\rangle
    + \langle \nabla v, \sigma_{n+1} \rangle & & \\
    \quad + k\llangle v, u_{n+1}\rrangle
    & = & \nonumber
    \frac{k-1}{k+1}
    \left(
    \langle v, k(\hat{\delta} + i) u_{n}\rangle
    - \langle \nabla v, \sigma_{n} \rangle
    + k\llangle v, u_{n}\rrangle\right) \\
    & &
    + \frac{2k}{k+1}R_u[v], \forall  v\in Q_h, \quad n=0,\ldots,mk-1,
    \label{eq:HSSpc sigma}
  \end{align}
  and $(\sigma_0,u_0)=(0,0)$.
\end{definition}
We will show that the operators $\tilde{B}_{\hat{\delta}}^{-1}B_0 - I$
and $\tilde{A}_{\hat{\delta}}^{-1}A_0 - I$ have spectral radius bounds
that are independent of $k$ {and can be made less than 1}, leading to $k$ independent convergence
rate bounds for the primal and mixed iterative schemes, respectively.

Our results rely upon the following two theorems which we quote
without proof, concerning the problem of finding $u\in H^1(\Omega)$
such that
\begin{equation}
  -(\hat{k}^2+i\epsilon)
  \langle v, u\rangle + \langle \nabla v, \nabla u \rangle
  - i\eta\llangle v, u \rrangle = \langle v, f \rangle
  + \llangle v, g \rrangle, \quad \forall v\in H^1(\Omega).
  \label{eq:gander}
\end{equation}
\begin{theorem}[Theorem 2.9 of \cite{gander2015applying}]
  Let $u\in H^1(\Omega)$ solve \eqref{eq:gander}, in a domain $\Omega$
  that is Lipschitz and star-shaped with respect to a ball. If
  $\epsilon \lesssim \hat{k}$, $\Re(\eta) \sim \hat{k}$, $\Im(\eta) \lesssim \hat{k}$,
  then given $k_0>0$, there exists $C>0$ independent of $\hat{k}$, $\epsilon$
  and $\eta$ such that
  \begin{equation}
    \|u\|_{1,\hat{k},\Omega}^2 := |u|_{1,\Omega}^2 + \hat{k}^2\|u\|^2_{0,\Omega}
    \leq C\left(\|f\|_{0,\Omega}^2+ \|g\|_{0,\partial\Omega}^2\right),
  \end{equation}
  for all $\hat{k}\geq k_0$.
  \label{th:g20152.9}
\end{theorem}
\begin{remark}
In this article we found it easier to work with $-(\delta - ik)^2$ instead
of $\hat{k}^2+i\epsilon$, hence we write $\hat{k} = \sqrt{k^2-\delta^2}$,
$\epsilon=2k\delta$. The relevant case for $\eta$ will be
$\eta = i\delta + k$. Hence, for $\delta$
independent of $k$ we have
\begin{equation}
  \|u\|_{1,{k},\Omega}^2 = |u|_{1,\Omega}^2 + {k}^2\|u\|^2_{0,\Omega}
  \lesssim \left(\|f\|_{0,\Omega}^2+ \|g\|_{0,\partial\Omega}^2\right),
\end{equation}
for all $k$ sufficiently large, where $a \lesssim b$ means that there exists a parameter independent
constant $C>0$ such that $a \lesssim Cb$.
\end{remark}
\begin{theorem}[Lemma 3.5 of \cite{gander2015applying}, adjusted to our
    parameters]
  \label{th:g20153.5}
  Let $u\in H^1(\Omega)$ solve \eqref{eq:gander}, and let $u_h\in Q_h$
  solve
  \begin{equation}
    \label{eq:u_h}
    \langle v, (\delta - ik)^2 u_h \rangle  + \langle \nabla v,
    \nabla u_h\rangle + \llangle v, (\delta - ik)u_h \rrangle
    = \langle v, f \rangle + \llangle v, g \rrangle,
    \quad \forall v \in Q_h.
  \end{equation}
  Under the assumptions of Theorem
  \ref{th:g20152.9}, if we also assume that $\Omega$ is a convex
  polygon\footnote{More general assumptions in Theorem
  \ref{th:g20152.9} are discussed in \cite{gander2015applying} but we
  limit that discussion here.}, $\delta \lesssim k$, $k \sim \sqrt{k^2 - \delta^2}$, and $\delta \lesssim \sqrt{k^2 - \delta^2}$, then given $k_0 > 0$ there exists $C_1,C_2>0$
  (independent of $h$,$k$, and $\delta$) such that
  \begin{equation}
    \|u_h-u\|_{1,k,\Omega} \leq{C_2}\inf_{v\in Q_h}\|u-v\|_{1,k,\Omega},
    \label{eq:u_h est}
  \end{equation}
  whenever $h\sqrt{k^2 -\delta^2}(k-\delta) \leq C_1$ and $\sqrt{k^2-\delta^2}\geq k_0$.
\end{theorem}
\begin{corollary}
  Let $u_h\in Q_h$ solve \eqref{eq:u_h}.
  Then, under the assumptions of Theorem \ref{th:g20153.5},
  \begin{equation}
    \label{eq:u_h coroll}
    \|u_h \|_{1,k,\Omega}^2 \lesssim \|f\|^2_{0,\Omega}
    + \|g\|^2_{0,\partial\Omega}.
  \end{equation}
\end{corollary}
\begin{proof}
  Choosing $v=0$ in \eqref{eq:u_h est},
  \begin{equation}
    \|u_h\|_{1,k,\Omega}^2 \lesssim \|u\|_{1,k,\Omega}^2
    \lesssim \|f\|^2_{0,\Omega}
    + \|g\|^2_{0,\partial\Omega},
  \end{equation}
  by Theorem \ref{th:g20152.9} and the reverse triangle inequality.
\end{proof}
Continuing with our agenda to show {for the primal formulation} that
$\|\tilde{B}^{-1}_{\hat{\delta}} B_0-I\| \leq c$ for some $c<1$, the
strategy is to write
\begin{align}
  \tilde{B}_{\hat{\delta}}^{-1}B_0 - I &
  = \tilde{B}_{\hat{\delta}}^{-1}B_{\hat{\delta}} - I + \tilde{B}_{\hat{\delta}}^{-1}B_0
  - \tilde{B}_{\hat{\delta}}^{-1}B_{\hat{\delta}}, \\
  &= \left(\tilde{B}_{\hat{\delta}}^{-1}B_{\hat{\delta}} - I\right) + \tilde{B}_{\hat{\delta}}^{-1}
  B_{\hat{\delta}} B_{\hat{\delta}}^{-1}B_0
  - \tilde{B}_{\hat{\delta}}^{-1}B_{\hat{\delta}}, \\
  &= \left(\tilde{B}_{\hat{\delta}}^{-1}B_{\hat{\delta}} - I\right) + \tilde{B}_{\hat{\delta}}^{-1}
  B_{\hat{\delta}} \left(B_{\hat{\delta}}^{-1}B_0 - I\right).
\end{align}
Then taking some operator norm $\|\cdot\|$, we have
\begin{align}
  \|\tilde{B}_{\hat{\delta}}^{-1}B_0 - I\|
  \leq
  & \|\tilde{B}_{\hat{\delta}}^{-1}B_{\hat{\delta}} - I\| + \|\tilde{B}_{\hat{\delta}}^{-1}
  B_{\hat{\delta}}\| \|B_{\hat{\delta}}^{-1}B_0 - I\|, \\
   \leq & \|\tilde{B}_{\hat{\delta}}^{-1}B_{\hat{\delta}} - I\| +
  \left(\|\tilde{B}_{\hat{\delta}}^{-1}B_{\hat{\delta}} - I\| + 1\right)
  \|B_{\hat{\delta}}^{-1}B_0 - I\|.
\end{align}
Hence, we just need to be able to make $\|B_{\hat{\delta}}^{-1}B_0-I\|$ and
$\|\tilde{B}_{\hat{\delta}}^{-1}B_{\hat{\delta}} - I\|$ small enough for this to hold.
We can make $\|\tilde{B}_{\hat{\delta}}^{-1}B_{\hat{\delta}} - I\|_{\Hprimal}$ small by taking
$\mathcal{O}(k)$ HSS iterations to form $\tilde{B}_{\hat{\delta}}^{-1}$, as we
have already shown in the previous section. Now we want to find
a bound for $\|B_{\hat{\delta}}^{-1}B_0-I\|_{\Hprimal}$ which we can control by reducing
${\hat{\delta}}$ (independently of $k$, otherwise the object of having
a $k$ independent multigrid preconditioner is defeated).

\cite{gander2015applying} provided techniques for showing that
$\|\hat{B}_{\hat{\delta}}^{-1}\hat{B}_0-I\|\lesssim \hat{\delta}$, using the usual
operator 2-norm for matrices, which we need to adapt to produce an
equivalent estimate in the $\|\cdot\|_{\Hprimal}$ norm.

First we give convergence results for the primal system shifted HSS
preconditioner.
\begin{theorem}
  Under the assumptions of Theorem \ref{th:g20153.5},
  we have
  \begin{equation}
    \label{eq:shift est p}
    \|{B}_{\hat{\delta}}^{-1}B_0 - I \|_{\Hprimal} \leq c\hat{\delta}^{1/2},
  \end{equation}
  for some constant $c>0$ that is independent of $k$ and
  $\hat{\delta}$ for the primal formulation of the Helmholtz problem.
\end{theorem}
\begin{proof}
For $u_0\in Q_h$,
  we define $u=(B_{\hat{\delta}}^{-1}B_0 - I)u_0$,
  such that
\begin{align}
    \label{eq:sigma shift primal}
    (-\hat{\delta} + ik)^2\langle v,u\rangle + \langle \nabla v, \nabla u \rangle -(-\hat{\delta} +ik) \llangle v, u \rrangle =& -\hat{\delta}^2 \langle v, u_0 \rangle + 2\hat{\delta} ki \langle v, u_0 \rangle \nonumber \\
    & - \hat{\delta} \llangle v, u_0 \rrangle,
  \quad \forall v \in Q_h.
\end{align}
Recall that
\begin{equation}
    \|u \|^2_{{\Hprimal}} := 2\hat{\delta} k \|u \|^2_{0,\Omega} + k\|u \|^2_{0,\partial \Omega}.
\end{equation}
We will derive bounds for $\|u \|^2_{0,\Omega}$ and $\|u \|^2_{0,\partial\Omega}$. To derive the bound for $\|u \|^2_{0,\Omega}$ we apply \eqref{eq:u_h coroll} to obtain
\begin{equation}
    \|u \|^2_{1,k,\Omega} \lesssim \hat{\delta}^2k^2\|u_0 \|^2_{0,\Omega} + \hat{\delta}^2 \| u_0\|^2_{0,\partial \Omega},
\end{equation}
assuming that $\hat{\delta} \leq k$. Now we multiply this expression by $\frac{\hat{\delta}}{k}$ to obtain
\begin{align}
    \hat{\delta} k \|u \|^2_{0,\Omega} &\lesssim \hat{\delta}^2 \left (\hat{\delta} k \|u_0 \|^2_{0,\Omega} + \frac{\hat{\delta}}{k}\|u_0 \|^2_{0,\partial\Omega}\right ), \\
    &\lesssim \hat{\delta}^2 \left (\hat{\delta} k \|u_0 \|^2_{0,\Omega} + \frac{\hat{\delta}^2}{k^3}\|u_0 \|^2_{0,\partial \Omega} + k \|u_0 \|^2_{0,\partial \Omega}\right ), \\
    &\lesssim \hat{\delta}^2 \left ( \hat{\delta} k \|u_0 \|^2_{0,\Omega} + k \|u_0 \|^2_{0,\partial \Omega}\right ), \\
    &\lesssim \hat{\delta}^2 \|u_0 \|^2_{{\Hprimal}},
\end{align}
where we have used Young's inequality in the second line. To derive the bound for $\|u \|^2_{0,\partial\Omega}$, we substitute $v=u$ in \eqref{eq:sigma shift primal} and take negative imaginary parts
\begin{align}
    2\hat{\delta} k \|u \|^2_{0,\Omega} + k\|u \|^2_{0,\partial\Omega} &= -2\hat{\delta} k\langle u,u_0 \rangle, \\
    &\leq 2\hat{\delta} k \|u_0 \|_{0,\Omega} \|u \|_{0,\Omega}, \\
    &\leq k\hat{\delta}^2 \|u_0 \|^2_{0,\Omega} + k \|u \|^2_{0,\Omega},
\end{align}
where we have used Schwarz's and Young's inequality in the second and third line respectively. Thus, using the bound for $\|u \|^2_{0,\Omega}$ we obtain
\begin{align}
    k \|u \|^2_{0,\partial \Omega} &\lesssim k \hat{\delta}^2 \|u_0 \|^2_{0,\Omega} + \hat{\delta} \|u_0 \|^2_{{\Hprimal}}, \\
    &\lesssim \hat{\delta} \|u_0 \|^2_{{\Hprimal}} + \hat{\delta} \|u_0 \|^2_{{\Hprimal}} \lesssim \hat{\delta} \|u_0 \|^2_{{\Hprimal}}.
\end{align}
So
\begin{align}
    \|u \|^2_{{\Hprimal}} &:= 2\hat{\delta} k \|u \|^2_{0,\Omega} + k\|u \|^2_{0,\partial \Omega}, \\
    &\lesssim \hat{\delta}^2 \|u_0 \|^2_{{\Hprimal}} + \hat{\delta} \|u_0 \|^2_{{\Hprimal}} \lesssim \hat{\delta} \|u_0 \|^2_{{\Hprimal}}.
\end{align}
Finally we may write
\begin{equation}
  \|B_{\hat{\delta}}^{-1}B_0 - I\|_{\Hprimal} = \sup_{0\neq u_0\in Q_h}
  \frac{\|(B_{\hat{\delta}}^{-1}B_0 - I)u_0\|_{\Hprimal}}{\|u_0\|_{\Hprimal}}
  =
  \sup_{0\neq u_0 Q_h}\frac{\|u\|_{\Hprimal}}{\|u_0\|_{\Hprimal}}
  \leq C\hat{\delta}^{1/2},
\end{equation}
for some constant $C>0$ independent of $k$ and $\hat{\delta}$. Hence, if we
choose $\hat{\delta}$ sufficiently small, we have $C\hat{\delta}^{1/2} < 1$, as
required.
\end{proof}

\begin{corollary}
  \label{cor:shifted HSS preconditioning primal}
  Under the assumptions of Theorem \ref{th:g20153.5},
  and for sufficiently large $m$ and $\hat{\delta}$ sufficiently small,
  there exists a constant $c<1$ independent of $k$, such that
  \begin{equation}
    \|\tilde{B}_{\hat{\delta}}^{-1}B_0 - I \|_{\Hprimal} \leq c,
  \end{equation}
  which implies that the preconditioned Richardson iteration converges
  at a rate independent of $k$.
\end{corollary}
\begin{proof}
    \begin{align} \nonumber
    \|\tilde{B}_{\hat{\delta}}^{-1}B_0 - I \|_{\Hprimal}
    \leq &\|\tilde{B}_{\hat{\delta}}^{-1}B_{\hat{\delta}} - I\|_{\Hprimal} \\
    & \quad + \left(\|\tilde{B}_{\hat{\delta}}^{-1}B_{\hat{\delta}} - I\|_{\Hprimal}+1\right)
    \|{B}_{\hat{\delta}}^{-1}B_0 - I\|_{\Hprimal}, \\
    \leq & e^{-m} + (e^{-m} + 1)C\hat{\delta}^{1/2},
  \end{align}
  having made use of Corollary \ref{eq:hss est} and \eqref{eq:shift est}. We
  can make this upper bound arbitrarily small by selecting large
  enough $m$ and small enough $\hat{\delta}$.
\end{proof}
\begin{remark}
  A consequence of this result is that the spectral radius of
  $\tilde{B}^{-1}_{\hat{\delta}} B_0-I$ is strictly less than 1,
  independent of $k$.  From \cite{kirby2010functional}, we know that
  the spectral radius of
  $\hat{\tilde{B}}^{-1}_{\hat{\delta}}\hat{B}_0-I$ is strictly less
  than 1, where $\hat{\tilde{B}}^{-1}_{\hat{\delta}}$ and $\hat{B}_0$
  are the matrices obtained by inserting basis functions into
  $\tilde{B}^{-1}_{\hat{\delta}}$ and $B_0$ respectively. This means
  that the decay rate of the classical GMRES upper bound estimate as a
  function of iteration is independent of $k$. However, our
  preconditioned operator is non-normal (due to the Sommerfeld
  condition), and we do not have a $k$ independent bound for the
  conditioning of the eigenmodes of
  $\hat{\tilde{B}}^{-1}_{\hat{\delta}}\hat{B}_0-I$, so this
  $k$-independent decay may not necessarily be useful. This is
  illustrated by the result of \cite{greenbaum1996any} showing that
  for any chosen nonincreasing GMRES convergence history, there exists
  a linear system that will exhibit it. We explore this aspect in our
  numerical experiments section.
\end{remark}
Now we extend these results to the mixed formulation.
\begin{theorem}
  Under the assumptions of Theorem \ref{th:g20153.5},
  we have
  \begin{equation}
    \label{eq:shift est}
    \|{A}_{{\hat{\delta}}}^{-1}A_0 - I \|_{\Hmixed} \leq c\hat{\delta}^{1/2},
  \end{equation}
  for some constant $c>0$ that is independent of $k$ and
  $\hat{\delta}$ for the mixed formulation of the Helmholtz problem.
\end{theorem}
\begin{proof}
  For $(\sigma_0, u_0)\in V_h\times Q_h$,
  we define $(\sigma,u)=(A_{\hat{\delta}}^{-1}A_0 - I)(\sigma_0,u_0)$,
  such that
  \begin{align}
    \label{eq:sigma shift}
  \langle \tau, (\hat{\delta} - ik) \sigma \rangle
  - \langle \tau, \nabla u \rangle & = -\hat{\delta}
  \langle \tau, \sigma_0 \rangle, \\
  \langle v, (\hat{\delta} - ik)u \rangle
  + \llangle v, u \rrangle + \langle \nabla v, \sigma \rangle
  & = -\hat{\delta} \langle v, u_0 \rangle.
\end{align}
Eliminating $\sigma$ as previously, we get
\begin{equation}
  \langle v, (\hat{\delta} - ik)^2 u \rangle
  + \llangle v, (\hat{\delta} - ik)u\rrangle
  + \langle \nabla v, \nabla u \rangle
  = -\hat{\delta} \langle v, (\hat{\delta} - ik) u_0 \rangle
  + \hat{\delta} \langle \nabla v, \sigma_0 \rangle,
  \quad \forall v \in Q_h.
  \label{eq:eliminated shift eqn}
\end{equation}
We would like to use \eqref{eq:u_h coroll}, but the presence of
$\nabla v$ in the right hand side of \eqref{eq:eliminated shift eqn}
prevents this. To remove it, define $\phi \in Q_h$ such that
\begin{equation}
  k^2\langle \phi, v \rangle
  + k^2\llangle \phi, v \rrangle
  + \langle \nabla \phi, \nabla v \rangle
  = \hat{\delta} \langle \nabla v, \sigma_0 \rangle, \quad
  \forall v \in Q_h.
\end{equation}
Then, choosing $v=\phi$, and using the Schwarz' and Young's
inequalities gives
\begin{equation}
  k^2\|\phi\|_{0,\Omega}^2 +
  k^2\|\phi\|_{0,\partial\Omega}^2 +
  |\phi|_{1,\Omega}^2
  \leq \frac{1}{2}|\phi|_{1,\Omega}^2 + \frac{\hat{\delta}^2}{2}
  \|\sigma_0\|_{0,\Omega},
\end{equation}
and hence,
\begin{equation}
  \label{eq:phi est}
  \|\phi\|_{1,k,\Omega,\partial\Omega}^2
  := k^2\|\phi\|_{0,\Omega}^2
  + k^2\|\phi\|_{0,\partial\Omega}^2
  + |\phi|_{1,\Omega}^2
  \leq \hat{\delta}^2 \|\sigma_0\|_{0,\Omega}^2.
\end{equation}
Defining $\hat{u}=u-\phi \in Q_h$, we obtain
\begin{align}
  \nonumber
  \langle v, (\hat{\delta} - ik)^2 \hat{u} \rangle
  + \llangle v, (\hat{\delta} - ik) \hat{u}\rrangle
  + \langle \nabla v, \nabla \hat{u} \rangle
  = & -\hat{\delta} \langle v, (\hat{\delta} - ik) u_0 \rangle \\
  &  + \left(k^2 - (\hat{\delta} - ik)^2\right)\langle v, \phi \rangle \nonumber
  \\
&  + \left(k^2 - (\hat{\delta} - ik)\right)\llangle v, \phi \rrangle,
  \quad \forall v \in Q_h.
  \label{eq:uhat with phi}
\end{align}
Then, we can use \eqref{eq:u_h coroll} to obtain
\begin{align}
  \|\hat{u}\|_{1,k,\Omega}^2
  & \lesssim \hat{\delta}^2k^2\|u_0\|_{0,\Omega}^2
  + k^4\|\phi\|_{0,\Omega}^2 + k^4\|\phi\|_{0,\partial\Omega}^2, \\
  & \lesssim \hat{\delta}^2k^2\|u_0\|_{0,\Omega}^2
  + k^2\|\phi\|_{0,k,\Omega,\partial\Omega}^2, \\
  & \lesssim \hat{\delta}^2k^2\|u_0\|_{0,\Omega}^2
  + \hat{\delta}^2k^2\|\sigma_0\|_{0,\Omega}^2,
  \label{eq:uhat est}
\end{align}
assuming that $\hat{\delta} < k$,
having also used \eqref{eq:phi est} in the last inequality.
We can then use \eqref{eq:uhat est} to estimate $\|u\|_{1,k,\Omega}$,
\begin{align}
  \|u\|_{1,k,\Omega}^2 & \leq \|\hat{u}\|_{1,k,\Omega}^2
  + \|\phi\|_{1,k,\Omega}^2, \\
  & \lesssim \hat{\delta}^2k^2\|u_0\|_{0,\Omega}^2
  + \hat{\delta}^2k^2\|\sigma_0\|_{0,\Omega}^2 + \hat{\delta}^2\|\sigma_0\|_{0,\Omega}^2, \\
  & \lesssim \hat{\delta}^2k^2\|u_0\|_{0,\Omega}^2
  + \hat{\delta}^2k^2\|\sigma_0\|_{0,\Omega}^2,
  \label{eq:u 1,k,Omega}
\end{align}
assuming that $k>1$.

Now we have an estimate for $(\sigma,u)$ in the $\|\cdot
\|_{1,k,\Omega}$ norm. To get an estimate in the $\|\cdot\|_{\Hmixed}$ norm,
we need to get further estimates for $\|u\|_{0,\partial\Omega}$
and $\|\sigma\|_{0,\Omega}$.

To estimate $\|u\|_{0,\partial\Omega}$, take $v=\hat{u}$
in \eqref{eq:uhat with phi}
and take negative imaginary parts
%, then use
%\eqref{eq:u 1,k,Omega},
to obtain
\begin{align}
  2\hat{\delta} k\|\hat{u}\|_{0,\Omega}^2
  + k\|\hat{u}\|_{0,\partial\Omega}^2
  \leq & \hat{\delta}|\hat{\delta}-ik|\|\hat{u}\|_{0,\Omega}\|u_0\|_{0,\Omega}
  \nonumber \\
  & + |k^2-(\hat{\delta}-ik)^2|\|\hat{u}\|_{0,\Omega}\|\phi\|_{0,\Omega}
  \nonumber \\
  & + |k^2-(\hat{\delta}-ik)|\|\hat{u}\|_{0,\partial\Omega}\|\phi\|_{0,\partial\Omega}, \\
  \nonumber
  \leq & \hat{\delta}|\hat{\delta}-ik| \times \nonumber \\
  & \quad \left(\frac{k}{2|\hat{\delta}-ik|\hat{\delta}}\|\hat{u}\|_{0,\Omega}^2
  +\frac{\hat{\delta}|\hat{\delta}-ik|}{2k}\|u_0\|_{0,\Omega}^2\right) \\
  & \nonumber + |k^2-(\hat{\delta}-ik)^2| \times \nonumber \\
  & \quad \left(\frac{k}{2|k^2-(\hat{\delta}-ik)^2|}
  \|\hat{u}\|_{0,\Omega}^2+
  \frac{|k^2-(\hat{\delta}-ik)^2|}{2k}
  \|\phi\|_{0,\Omega}^2\right) \\
  & + |k^2-(\hat{\delta}-ik)| \times \nonumber \\
  & \quad \left(\frac{k}{2|k^2-(\hat{\delta}-ik)|}
  \|\hat{u}\|_{0,\partial\Omega}^2
  +\frac{|k^2-(\hat{\delta}-ik)|}{2k}
  \|\phi\|_{0,\partial\Omega}^2\right), \\
  \leq & k\|\hat{u}\|_{0,\Omega}^2 + \frac{\hat{\delta}^2|\hat{\delta}-ik|^2}
       {2k}\|u_0\|_{0,\Omega}^2
       + \frac{|k^2-(\hat{\delta}-ik)^2|^2}{2k}\|\phi\|_{0,\Omega}^2        \nonumber
       \\
& \qquad       + \frac{k}{2}\|\hat{u}\|_{0,\partial\Omega}^2+
       \frac{|k^2-(\hat{\delta}-ik)|^2}{2k}\|\phi\|_{0,\partial\Omega}^2
\end{align}
so
\begin{align}
  k\|\hat{u}\|_{0,\partial\Omega}^2 \lesssim &
  k\|\hat{u}\|^2_{0,\Omega} + \hat{\delta}^2k \|u_0\|^2_{0,\Omega}
  + k^3\|\phi\|_{0,\Omega}^2
  + k^3\|\phi\|_{0,\partial\Omega}^2 \\
  \lesssim &
  \frac{1}{k}\underbrace{\|\hat{u}\|^2_{1,k,\Omega}}_{\lesssim
    \hat{\delta}^2k^2(\|u_0\|^2_{0,\Omega}
    + \|\sigma_0\|^2_{0,\Omega})}
    + \hat{\delta}^2k \|u_0\|^2_{0,\Omega}
  + k\underbrace{\|\phi\|_{1,k,\partial,\Omega}^2}_{\leq \hat{\delta}^2\|\sigma_0\|_{0,\Omega}^2} \\
  \lesssim & k\hat{\delta}^2\left(\|u_0\|_{0,\Omega}^2 +
  \|\sigma_0\|_{0,\Omega}^2\right),
\end{align}
having used \eqref{eq:phi est} and \eqref{eq:u 1,k,Omega}.

Hence by \eqref{eq:phi est}, the triangle and Young's inequalities
\begin{align}
  \label{eq:u partial Omega est}
  \|u\|_{0,\partial\Omega}^2 = \|\hat{u}+\phi\|_{0,\partial\Omega}^2 \lesssim & \|\hat{u}\|_{0,\partial\Omega}^2 + \|\phi\|_{0,\partial\Omega}^2 \\
  \lesssim & \hat{\delta}^2
  \left(
  \|u_0\|_{0,\Omega}^2 + \|\sigma_0\|_{0,\Omega}^2\right).
\end{align}

To estimate $\|\sigma\|_{0,\Omega}^2$, use $\tau=\sigma$ in
\eqref{eq:sigma shift}, leading to
\begin{align}
  (\hat{\delta} - ik)\|\sigma\|_{0,\Omega}^2 =
  -\hat{\delta} \langle \sigma, \sigma_0 \rangle
  + \langle \sigma, \nabla u \rangle.
\end{align}
After taking absolute values, and using Schwarz' and Young's
inequalities, we have
\begin{align}
  k\|\sigma \|^2_{0,\Omega} \leq& \hat{\delta}\|\sigma_0\|_{0,\Omega}
  \|\sigma\|_{0,\Omega}
  + |u|_{1,\Omega}\|\sigma\|_{0,\Omega}, \\
  \leq & \hat{\delta} \left(\frac{k}{4\hat{\delta}}\|\sigma\|_{0,\Omega}^2
  + \frac{2\hat{\delta}}{2k}\|\sigma_0\|_{0,\Omega}^2\right)
  + \frac{k}{4}\|\sigma\|_{0,\Omega}^2 + \frac{2}{2k}|u|_{1,\Omega}^2,
\end{align}
so
\begin{align}
  k\|\sigma\|_{0,\Omega}^2 \lesssim
  & \frac{1}{k}\|\sigma_0\|_{0,\Omega}^2
  + \frac{1}{k}|u|_{1,\Omega}^2, \\
  \lesssim &
  \frac{1}{k} \|\sigma_0\|^2_{0,\Omega}
  + \hat{\delta}^2 k\left(\|\sigma_0\|^2_{0,\Omega} + \|u_0\|^2_{0,\Omega}\right).
\end{align}
having used \eqref{eq:u 1,k,Omega}. Thus,
\begin{equation}
  \|\sigma\|_{0,\Omega}^2 \lesssim \hat{\delta}^2\left(\|\sigma_0\|^2_{0,\Omega} + \|u_0\|^2_{0,\Omega}\right).
  \label{eq:sigma est}
\end{equation}
Now we combine \eqref{eq:u 1,k,Omega} with \eqref{eq:u partial Omega
  est} and \eqref{eq:sigma est} to get
\begin{align}
  \|(\sigma,u)\|_H^2
  & = \hat{\delta} \|\sigma\|_{0,\Omega}^2 + \hat{\delta} \|u\|_{0,\Omega}^2
  + \|u\|_{0,\partial\Omega}^2, \\
  & \leq \hat{\delta} \|\sigma\|_{0,\Omega}^2 +
  \frac{\hat{\delta}}{k^2}\|u\|_{1,k,\Omega}^2
  + \|u\|_{0,\partial\Omega}^2, \\
  & \leq \hat{\delta}^3\left(\|u_0\|_{0,\Omega}^2 + \|\sigma_0\|_{0,\Omega}^2\right)
  + \hat{\delta}^3\left(\|u_0\|^2_{0,\Omega} + \|\sigma_0\|^2_{0,\Omega}\right)
  + \hat{\delta}^2\left(\|u_0\|^2_{0,\Omega} + \|\sigma_0\|^2_{0,\Omega}\right)\\
  & \lesssim \hat{\delta}\|(\sigma_0,u_0)\|_H^2.
\end{align}

Finally we may write
\begin{align}
  \|A_{\hat{\delta}}^{-1}A_0 - I\|_{\Hmixed} &= \sup_{0\neq(\sigma_0,u_0)\in V_h\times Q_h}
  \frac{\|(A_{\hat{\delta}}^{-1}A_0 - I)(\sigma_0,u_0)\|_{\Hmixed}}{\|(\sigma_0,u_0)\|_{\Hmixed}}, \\
  & =
  \sup_{0\neq(\sigma_0,u_0)\in V_h\times Q_h}\frac{\|(\sigma,u)\|_{\Hmixed}}{\|(\sigma_0,u_0)\|_{\Hmixed}}, \\
  & \leq C\hat{\delta}^{1/2},
\end{align}
for some constant $C>0$ independent of $k$ and $\hat{\delta}$. Hence, if we
choose $\hat{\delta}$ sufficiently small, we have $C\hat{\delta}^{1/2} < 1$, as
required.
\end{proof}

\begin{corollary}
\label{cor:shifted HSS preconditioning}
  Under the assumptions of Theorem \ref{th:g20153.5},
  and for sufficiently large $m$ and $\hat{\delta}$ sufficiently small,
  there exists a constant $c<1$ independent of $k$, such that
  \begin{equation}
    \|\tilde{A}_{\hat{\delta}}^{-1}A_0 - I \|_{\Hmixed} \leq c,
  \end{equation}
  which implies that the preconditioned Richardson iteration converges
  at a rate independent of $k$.
\end{corollary}
\begin{proof}
    The proof is identical to Corollary \ref{cor:shifted HSS
      preconditioning primal} for the primal case.
\end{proof}

\section{Numerical examples}\label{sec:numerics}

\begin{figure}
\centering
\includegraphics[width=0.4\linewidth]{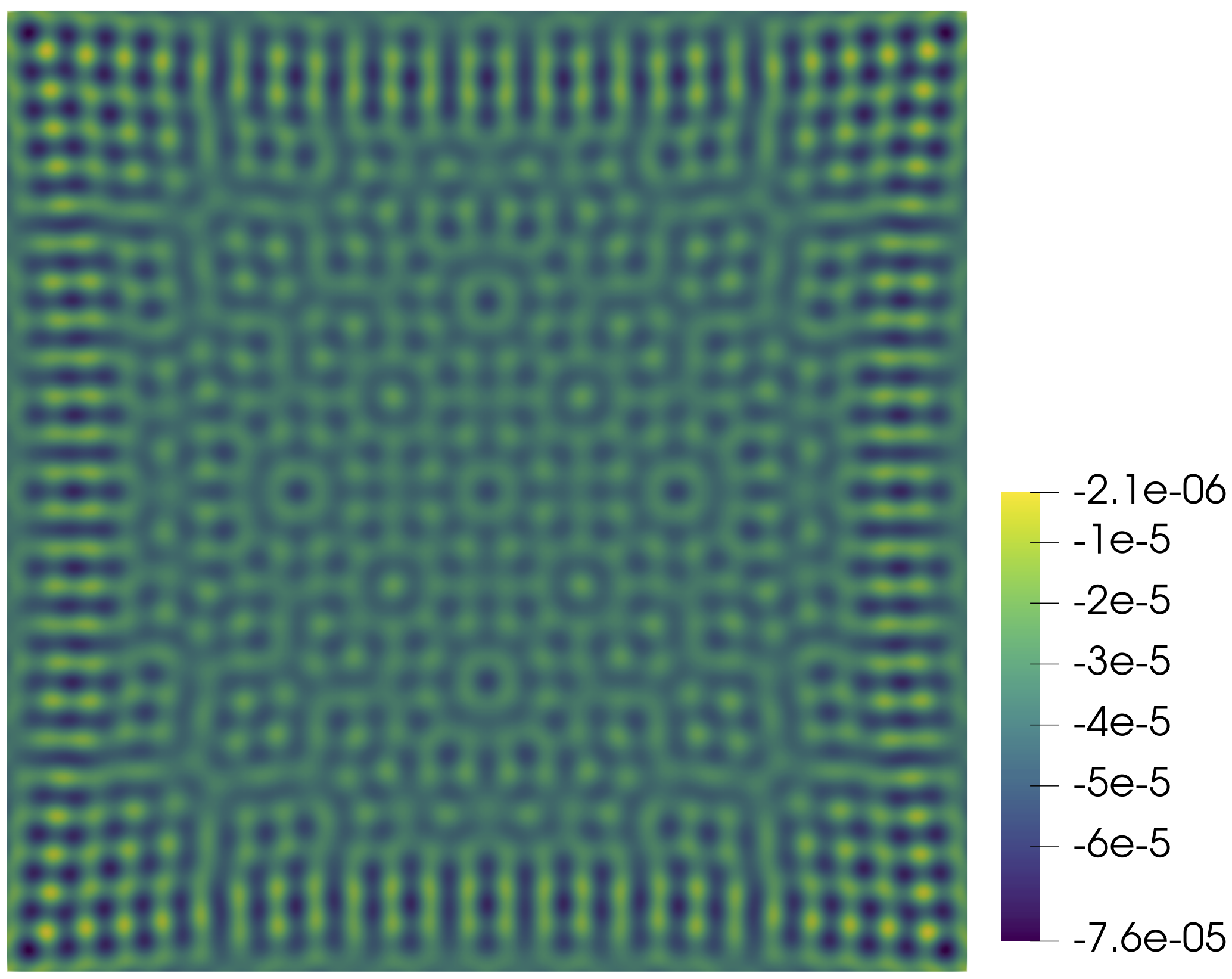}
\caption{Solution of the uniform source case when \(k=160\).}
\label{fig:uniform-source:solution}
\end{figure}

\begin{figure}
\centering
\includegraphics[width=0.4\linewidth]{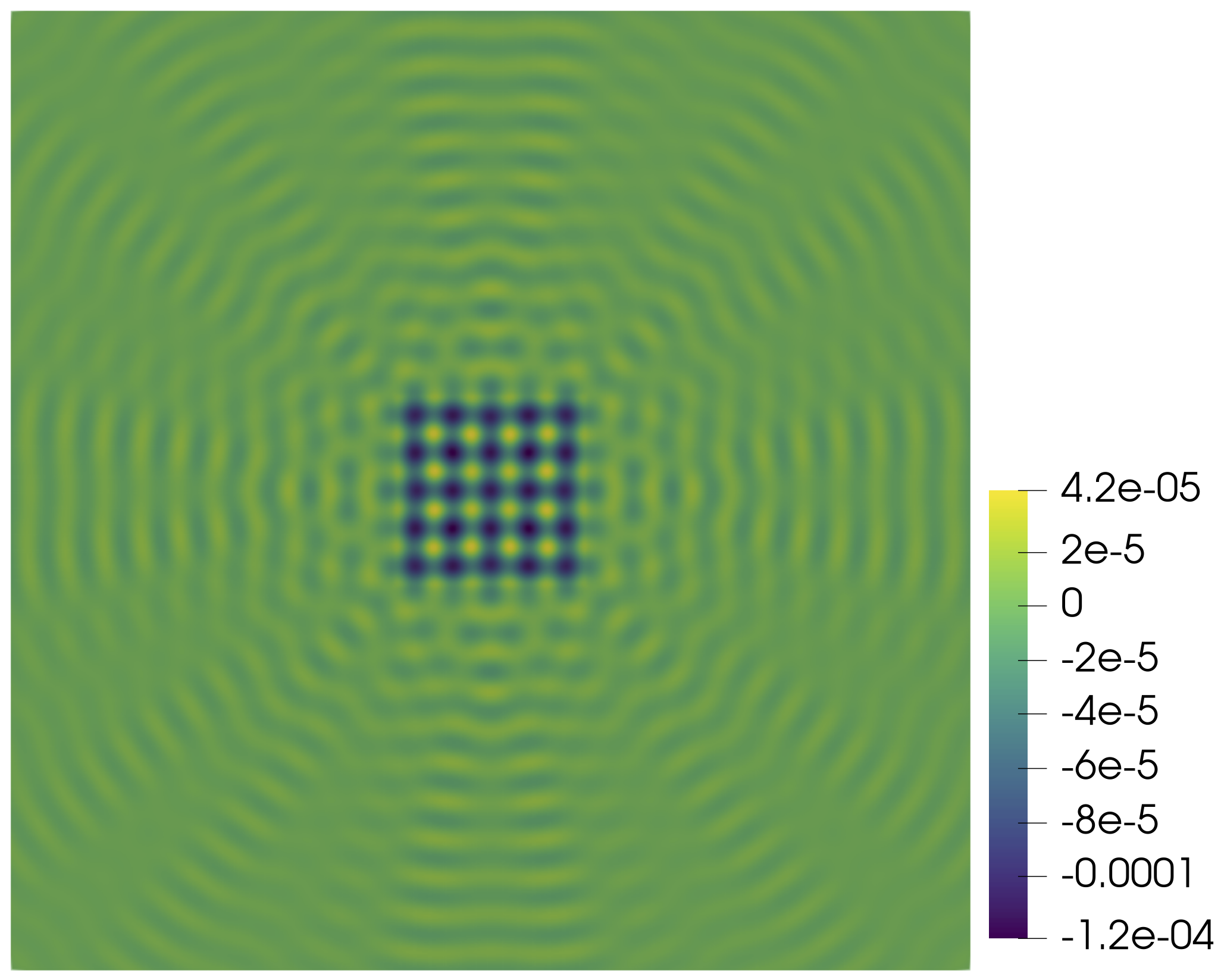}
\caption{Solution of the box source case when \(k=160\).}
\label{fig:box-source:solution}
\end{figure}

%% The preconditioning strategy depends on three key considerations:
%% \begin{itemize}
%% \item \textbf{$A_\delta$ is a good preconditioner for $A_0$.} \cite{gander2015applying} (maybe other citations are better?) showed that $A_\delta$ gives $k$-robust convergence.
%% \item \textbf{The HSS iteration is a good approximation for $A_\delta$.} We will demonstrate this by replacing $A_\delta$ with $m$ HSS iterations and solving every iteration directly to remove other approximation errors. We will show that k-independent GMRES convergence is obtained when $m$ is at least $\mathcal{O}(k)$.
%% \item \textbf{Multigrid methods effectively solve the HSS iteration.} We will demonstrate this by replacing each HSS iteration by $n\leq r$ multigrid sweeps with a standard point smoother. We will show that $k$-independent GMRES convergence is obtained when $r$ is sufficiently large but $k$-independent.
%% \end{itemize}
%% For two numerical examples, the uniform source and the box source, the above considerations are tested.

In this section we provide some results of numerical experiments that
demonstrate our theoretical results, as well as going beyond the
theory to demonstrate the practicality of the approach when the
shifted operator is solved approximately using multigrid with standard
transfer operators and smoother.

\subsection{Problem setups}
In our numerical experiments, we consider two idealised test problems
to illustrate our results. The first problem is the uniform source (example 5.2 from \cite{gander2015applying}) problem, solving
\begin{align}
    -k^2u - \nabla ^2u = 1
\end{align}
in the unit square with Sommerfeld boundary conditions on the entire
boundary. A plot of the numerical solution is given in Figure
\ref{fig:uniform-source:solution}.

The other problem we consider is the ``box source'' problem,
\begin{align}
    -k^2u - \nabla ^2u = \begin{cases}
        1 & \text{if }(x,y) \in [0.4, 0.6]\times[0.4, 0.6], \\
        0 & \text{otherwise},
    \end{cases}
\end{align}
solved in the unit square with Sommerfeld boundary conditions on the
entire boundary. An example plot of the numerical solution is given
in Figure \ref{fig:box-source:solution}.

Unless stated otherwise, the mesh is uniformly triangular with $c_0
k^{\frac{3}{2}}$ cells in every direction to account for the pollution
effect, with a constant scaling parameter \(c_0\) with default value
1. For the mixed and primal formulation we use $Q_{h}\times
V_{h}=[DG_0]^2\times CG_1$ and \(V_{h}=CG_1\), respectively. We
consider both preconditioned Richardson iteration (the case for which
we have theoretical results) and the FGMRES Krylov method
\cite{saad1993flexible}.  In all numerical examples the
preconditioning strategy is based on $\hat{\delta} = 2$, and the outer
iteration (either the fixed point Richardson iteration or the FGMRES
Krylov method) has a relative tolerance of $10^{-6}$.  As recommended
in \cite{cocquet2017large}, a random initial guess is used so that all
frequencies appear in the initial error, which ensures a realistic
indication of the overall convergence rate.

All numerical results are obtained using Firedrake and PETSc
(\cite{petsc-user-ref}, \cite{petsc-efficient}, \cite{Dalcin2011},
\cite{FiredrakeUserManual}, \cite{Kirby2017}).
The Python scripts are available at \cite{knook_2025}.

\subsection{Direct solution of the HSS inner problem}

\begin{table}[t]
\centering
\makebox[\textwidth][c]{
\begin{tabular}{ccccccc}
    \toprule
      & \multicolumn{3}{c}{Uniform source}      & \multicolumn{3}{c}{Box source} \\
        \cmidrule(rl){2-4}                        \cmidrule(rl){5-7}
\(k\) & \(\theta={1/2}\) & \(\theta=1\) & \(\theta={3/2}\) & \(\theta={1/2}\) & \(\theta=1\)  & \(\theta={3/2}\) \\
\midrule
 16  & 114 (456)  & 16 (256)  & 11 (704)   & 114 (456)  & 16 (256)  & 11 (704)   \\
 32  & 144 (864)  & 14 (448)  & 10 (1820)  & 144 (864)  & 14 (448)  & 10 (1820)  \\
 64  & 246 (1968) & 15 (960)  & 9  (4608)  & 247 (1976) & 15 (960)  & 9  (4608)  \\
128  & 430 (5160) & 15 (1920) & 9  (13041) & 431 (5172) & 15 (1920) & 9  (13041) \\
\bottomrule
\end{tabular}
}
\caption{Number of Richardson iterations to solve the primal formulation when \(B_{\hat{\delta}}\) is approximated with \(k^{\theta}\) HSS iterations, and each HSS iteration is solved directly. The total number of HSS iterations are shown in brackets.}
\label{tab:primal:hss-direct-its-richardson}

\makebox[\textwidth][c]{
\begin{tabular}{ccccccc}
    \toprule
      & \multicolumn{3}{c}{Uniform source}      & \multicolumn{3}{c}{Box source} \\
        \cmidrule(rl){2-4}                        \cmidrule(rl){5-7}
\(k\) & \(\theta={1/2}\) & \(\theta=1\) & \(\theta={3/2}\) & \(\theta={1/2}\) & \(\theta=1\)  & \(\theta={3/2}\) \\
\midrule
 16  & 29 (116) & 8 (128) & 6 (384)  & 29 (116) & 8 (128) & 6 (384)  \\
 32  & 38 (228) & 6 (192) & 4 (724)  & 38 (228) & 6 (192) & 4 (724)  \\
 64  & 49 (392) & 6 (384) & 2 (1024) & 50 (400) & 6 (384) & 2 (1024) \\
128  & 60 (720) & 6 (768) & 2 (2896) & 60 (720) & 6 (768) & 2 (2896) \\
\bottomrule
\end{tabular}
}
\caption{Number of FGMRES iterations to solve the primal formulation when \(B_{\hat{\delta}}\) is approximated with \(k^{\theta}\) HSS iterations, and each HSS iteration is solved directly. The total number of HSS iterations are shown in brackets.}
\label{tab:primal:hss-direct-its}
\end{table}

\begin{table}[t]
\centering
\makebox[\textwidth][c]{
\begin{tabular}{ccccccc}
    \toprule
      & \multicolumn{3}{c}{Uniform source}      & \multicolumn{3}{c}{Box source} \\
        \cmidrule(rl){2-4}                        \cmidrule(rl){5-7}
\(k\) & \(\theta={1/2}\) & \(\theta=1\) & \(\theta={3/2}\) & \(\theta={1/2}\) & \(\theta=1\)  & \(\theta={3/2}\) \\
\midrule
 16  &  63  (252) & 13  (208) & 10   (640) &  63  (252) & 13  (208) & 11   (704) \\
 32  & 123  (738) & 14  (448) & 10  (1820) & 122  (732) & 14  (448) & 10  (1820) \\
 64  & 224 (1792) & 14  (896) &  9  (4608) & 224 (1792) & 14  (896) &  9  (4608) \\
128  & 380 (4560) & 14 (1792) &  9 (13041) & 380 (4560) & 14 (1792) &  9 (13041) \\
\bottomrule
\end{tabular}
}
\caption{Number of Richardson iterations to solve the mixed formulation when \(A_{\hat{\delta}}\) is approximated with \(k^{\theta}\) HSS iterations, and each HSS iteration is solved directly. The total number of HSS iterations are shown in brackets.}
\label{tab:mixed:hss-direct-its-richardson}

\makebox[\textwidth][c]{
\begin{tabular}{ccccccc}
    \toprule
      & \multicolumn{3}{c}{Uniform source}      & \multicolumn{3}{c}{Box source}          \\
        \cmidrule(rl){2-4}                         \cmidrule(rl){5-7}
\(k\) & \(\theta={1/2}\) & \(\theta=1\) & \(\theta={3/2}\) & \(\theta={1/2}\) & \(\theta=1\) & \(\theta={3/2}\) \\
\midrule
 16  & 39  (156) & 9 (144) & 7  (448) & 39  (156) & 9 (144) & 7  (448) \\
 32  & 49  (294) & 8 (256) & 6 (1086) & 48  (288) & 8 (256) & 6 (1086) \\
 64  & 73  (584) & 8 (512) & 5 (2560) & 71  (568) & 8 (512) & 5 (2560) \\
128  & 84 (1008) & 7 (896) & 3 (4344) & 85 (1020) & 7 (896) & 3 (4344) \\
\bottomrule
\end{tabular}
}
\caption{Number of FGMRES iterations to solve the mixed formulation when \(A_{\hat{\delta}}\) is approximated with \(k^{\theta}\) HSS iterations, and each HSS iteration is solved directly. The total number of HSS iterations are shown in brackets.}
\label{tab:mixed:hss-direct-its}
\end{table}

Here, we provide numerical results that demonstrate the behaviour
prescribed by {Corollaries} \ref{cor:shifted HSS preconditioning primal} and
\ref{cor:shifted HSS preconditioning}. This is done by using an
iterative method to solve the mixed and primal problems (Richardson or FGMRES
iterations), using the preconditioners in Definitions \ref{def:primal shifted HSS}
and \ref{def:shifted HSS preconditioning mixed},
respectively. In both cases, the linear system arising in each HSS iteration is
solved using a direct solver \cite{MUMPS:1,MUMPS:2}.
We use FGMRES as the outer Krylov method despite the preconditioner being constant
to be consistent with the results in the next section where we use GMRES
in the multigrid approximation of the HSS operator, requiring the
use of a flexible outer method.

The results for the primal system are shown in
Tables \ref{tab:primal:hss-direct-its-richardson} and \ref{tab:primal:hss-direct-its},
and the results for the mixed system are shown in Tables
\ref{tab:mixed:hss-direct-its-richardson} and \ref{tab:mixed:hss-direct-its}.
The tables show the number of outer iterations required to reach the required
relative tolerance \emph{versus} $k$, with $\ceil{k^\theta}$ inner HSS iterations
in the preconditioners, for $\theta\in \{1/2,1,3/2\}$ (equivalently
$m\in \{\ceil{k^{-1/2}}, 1, \ceil{k^{1/2}}\}$).
We observe the prescribed $k$ independence in the outer iterations when
$\theta=1$. For $\theta=1/2$, the required number of iterations grows
with $k$, whilst for $\theta=3/2$, the required number of iterations
reduces as $k$ increases. This demonstrates that our requirement of
scaling the inner iterations in proportion to $k$ is necessary and
sufficient, as expected. The slight decline in the number of
iterations for $\theta=1$ with increasing $k$ is expected because the
Bernoulli estimate is only sharp in the limit $k\to \infty$.

Also in Tables \ref{tab:primal:hss-direct-its} and \ref{tab:mixed:hss-direct-its}
are the total number of HSS iterations across all outer iterations. Although for
$\theta=1/2$ the number of outer iterations increases with $k$, with FGMRES they
do so at a rate of roughly $\mathcal{O}(k^{1/2})$, at least for the range
of $k$ tested. This means that, in practice, the total number of HSS iterations
across all outer FGMRES iterations for $\theta=1/2$ is not much worse than the total
number for $\theta=1$ (and in a couple of instances actually slightly better),
even though our theory does not provide any guarantees for this case.
On the other hand, for $\theta=3/2$ the rate at which the number of outer iterations decreases with $k$ is not fast enough to counteract the growth in the number of HSS iterations per outer iteration, so the total number of HSS iterations is significantly higher than for $\theta=1/2$ or $\theta=1$.

The measured average convergence rate of the HSS iterations as an approximation
to the shifted operator \(A_{\hat{\delta}}\) with $\theta=1$ is shown in Table
\ref{tab:hss-direct-rate}, and compared to the theoretical bounds in Propositions
\ref{prop:hss primal} and \ref{prop:hss mixed}, respectively. We see that the
bounds are very tight in all cases.  Note that the bounds in Propositions
\ref{prop:hss primal} and \ref{prop:hss mixed} apply to the \textit{error}
contraction rates but the results in Table \ref{tab:hss-direct-rate} are the
\textit{residual} contraction rates, which is what is available
during the iterations, although the bounds appear to apply just as well
to the residual.
Finally, Tables \ref{tab:shift-direct-rate-richardson} and
\ref{tab:shift-direct-rate} show the measured average convergence rate of the
outer iterations in the case of $\theta=1$. We again see asymptotically
$k$-independent stable convergence rates with increasing $k$ for both the
Richardson and FGMRES iterations.

\begin{table}
\centering
\makebox[\textwidth][c]{
\begin{tabular}{cccccc}
    \toprule
    & & \multicolumn{4}{c}{\(\eta_{s}\)} \\
        \cmidrule(rl){3-6}
    & & \multicolumn{2}{c}{Primal}      & \multicolumn{2}{c}{Mixed} \\
        \cmidrule(rl){3-4}                \cmidrule(rl){5-6}
\(k\) & \(\nu_{s}\) & Uniform & Box & Uniform & Box  \\
\midrule
 16  & 0.8824 & 0.8789 & 0.8789 & 0.8804 & 0.8803 \\
 32  & 0.9394 & 0.9382 & 0.9381 & 0.9372 & 0.9372 \\
 64  & 0.9692 & 0.9687 & 0.9687 & 0.9686 & 0.9687 \\
128  & 0.9845 & 0.9847 & 0.9847 & 0.9841 & 0.9841 \\
\bottomrule
\end{tabular}
}
\caption{Convergence rate of the HSS iterations for the primal and mixed formulations with each source type when \(B_{\hat\delta}\) and \(A_{\hat{\delta}}\) are approximated with \(k\) HSS iterations and each HSS iteration is solved directly. Expected rate \(\nu_{s}\) from Propositions \ref{prop:hss primal} and \ref{prop:hss mixed} and measured rates \(\eta_{s}\).}
\label{tab:hss-direct-rate}

\centering
\makebox[\textwidth][c]{
\begin{tabular}{ccccc}
    \toprule
    & \multicolumn{4}{c}{\(\eta_{h}\)} \\
      \cmidrule(rl){2-5}
    & \multicolumn{2}{c}{Primal}      & \multicolumn{2}{c}{Mixed} \\
        \cmidrule(rl){2-3}              \cmidrule(rl){4-5}
\(k\) & Uniform & Box & Uniform & Box  \\
\midrule
 16  & 0.4955 & 0.4960 & 0.4317 & 0.4322 \\
 32  & 0.4592 & 0.4596 & 0.4662 & 0.4663 \\
 64  & 0.4813 & 0.4819 & 0.4636 & 0.4645 \\
128  & 0.4781 & 0.4780 & 0.4602 & 0.4609 \\
\bottomrule
\end{tabular}
}
\caption{Convergence rate \(\eta_{h}\) of the Richardson iteration for the primal and mixed formulations with each source type when \(B_{\hat\delta}\) and \(A_{\hat{\delta}}\) are approximated with \(k\) HSS iterations and each HSS iteration is solved directly.}
\label{tab:shift-direct-rate-richardson}

\centering
\makebox[\textwidth][c]{
\begin{tabular}{ccccc}
    \toprule
    & \multicolumn{4}{c}{\(\eta_{h}\)} \\
      \cmidrule(rl){2-5}
    & \multicolumn{2}{c}{Primal}      & \multicolumn{2}{c}{Mixed} \\
        \cmidrule(rl){2-3}                \cmidrule(rl){4-5}
\(k\) & Uniform & Box & Uniform & Box  \\
\midrule
 16  & 0.2271 & 0.2265 & 0.2806 & 0.2815 \\
 32  & 0.1410 & 0.1404 & 0.2390 & 0.2386 \\
 64  & 0.1234 & 0.1232 & 0.2090 & 0.2081 \\
128  & 0.1163 & 0.1162 & 0.1790 & 0.1791 \\
\bottomrule
\end{tabular}
}
\caption{Convergence rate \(\eta_{h}\) of the outer FGMRES iterations for the primal and mixed formulations with each source type when \(B_{\hat\delta}\) and \(A_{\hat{\delta}}\) are approximated with \(k\) HSS iterations and each HSS iteration is solved directly.}
\label{tab:shift-direct-rate}
\end{table}

\subsection{Iterative solution of the HSS inner problem}
In this section we go beyond our theoretical results and examine
convergence of the fully scalable algorithm when the linear
system arising in each HSS iteration is solved
approximately using one or more multigrid cycles with Jacobi
smoothers.

The multigrid setup is very similar between the primal and the mixed formulations.
We use 4 levels with a coarsening factor of $2\times$, resulting in $64\times$
more DoFs on the finest level than the coarsest level. On each level, including
the coarsest, we use GMRES iterations with a standard Jacobi preconditioner as
the smoother. For the primal formulation, a single W-cycle per HSS iteration with
5 GMRES iterations per level was found to give robust outer iteration counts, and is used
in all further results with this formulation. Empirically, we found that the mixed
formulation required a more accurate approximation of the HSS operator for robust
iteration counts, so we use two W-cycles per HSS iteration with 4 GMRES iterations
per level for all further results with this formulation. In the mixed formulation,
the $\mathcal{O}(k)$ shifted operator is obtained \emph{after} eliminating the
vector-valued field, so multigrid is applied to a scalar-valued field of the same
size in both the primal and mixed formulations.

Tables \ref{tab:primal:hss-mg-its} and
\ref{tab:mixed:hss-mg-its} show the required number of Krylov
iterations for both test problems at various $k$ and $\theta$ values,
confirming that we achieve our aspiration to asymptotically $k$
independent convergence rates for $\theta=1$, similar to the previous
experiments with directly solving each HSS iteration in Tables \ref{tab:primal:hss-direct-its} and \ref{tab:mixed:hss-direct-its}. This is
confirmed by Tables \ref{tab:hss-mg-rate} and \ref{tab:shift-mg-rate},
which show measured convergence rates for the HSS iteration and the outer
FGMRES iteration in the same manner as Tables \ref{tab:hss-direct-rate}
and \ref{tab:shift-direct-rate}. The theoretical bound on the HSS iteration
is still very tight in the multigrid approximated case, and we see again
confirmation of the asymptotic $k$ independence for the outer solver.

\begin{table}
\centering
\makebox[\textwidth][c]{
\begin{tabular}{ccccccc}
    \toprule
      & \multicolumn{3}{c}{Uniform source}      & \multicolumn{3}{c}{Box source}           \\
        \cmidrule(rl){2-4}                         \cmidrule(rl){5-7}
\(k\) & \(\theta={1/2}\) & \(\theta=1\)  & \(\theta={3/2}\) & \(\theta={1/2}\) & \(\theta=1\) & \(\theta={3/2}\) \\
\midrule
 16 & 29 (116) & 8 (128) & 6  (384) & 29 (116) & 8 (128) & 6  (384) \\
 32 & 38 (228) & 6 (192) & 4  (720) & 38 (228) & 6 (192) & 4  (720) \\
 64 & 55 (440) & 6 (384) & 2 (1024) & 55 (440) & 6 (384) & 2 (1024) \\
128 & 66 (792) & 6 (768) & 2 (2896) & 66 (792) & 6 (768) & 2 (2896) \\
\bottomrule
\end{tabular}
}
\caption{Number of FGMRES iterations to solve the primal formulation when \(B_{\hat{\delta}}\) is approximated with \(k^{\theta}\) HSS iterations, and each HSS iteration is approximated with a single W cycle. Total number of multigrid cycles are shown in brackets.}
\label{tab:primal:hss-mg-its}

\centering
\makebox[\textwidth][c]{
\begin{tabular}{ccccccc}
    \toprule
      & \multicolumn{3}{c}{Uniform source}      & \multicolumn{3}{c}{Box source}           \\
        \cmidrule(rl){2-4}                         \cmidrule(rl){5-7}
\(k\) & \(\theta={1/2}\) & \(\theta=1\)  & \(\theta={3/2}\) & \(\theta={1/2}\) & \(\theta=1\)  & \(\theta={3/2}\) \\
\midrule
 16  & 38  (304) & 9  (288) & 7  (896) & 38  (304) & 9  (288) & 7  (896) \\
 32  & 50  (600) & 8  (512) & 6 (2172) & 49  (588) & 8  (512) & 6 (2172) \\
 64  & 72 (1152) & 8 (1024) & 5 (5120) & 72 (1152) & 8 (1024) & 5 (5120) \\
128  & 85 (2040) & 7 (1792) & 3 (8688) & 85 (2040) & 7 (1792) & 3 (8688) \\
\bottomrule
\end{tabular}
}
\caption{Number of FGMRES iterations to solve the mixed formulation when \(A_{\hat{\delta}}\) is approximated with \(k^{\theta}\) HSS iterations, and each HSS iteration is approximated with two W cycles. Total number of multigrid cycles are shown in brackets.}
\label{tab:mixed:hss-mg-its}
\end{table}

\begin{table}
\centering
\makebox[\textwidth][c]{
\begin{tabular}{cccccc}
    \toprule
    & & \multicolumn{4}{c}{\(\eta_{s}\)} \\
        \cmidrule(rl){3-6}
    & & \multicolumn{2}{c}{Primal}      & \multicolumn{2}{c}{Mixed} \\
        \cmidrule(rl){3-4}                \cmidrule(rl){5-6}
\(k\) & \(\nu_{s}\) & Uniform & Box    & Uniform & Box  \\
\midrule
 16  & 0.8824 & 0.8778 & 0.8777 & 0.8806 & 0.8803 \\
 32  & 0.9394 & 0.9376 & 0.9376 & 0.9373 & 0.9372 \\
 64  & 0.9692 & 0.9681 & 0.9681 & 0.9687 & 0.9687 \\
128  & 0.9845 & 0.9844 & 0.9844 & 0.9842 & 0.9842 \\
\bottomrule
\end{tabular}
}
\caption{Convergence rate of the HSS iterations for the primal and mixed formulations with each source type when \(B_{\hat\delta}\) and \(A_{\hat{\delta}}\) are approximated with \(k\) HSS iterations and each HSS iteration is approximated with one or two W cycles for the primal or mixed formulations respectively. Expected rate \(\nu_{s}\) and measured rates \(\eta_{s}\).}
\label{tab:hss-mg-rate}

\centering
\makebox[\textwidth][c]{
\begin{tabular}{ccccc}
    \toprule
    & \multicolumn{4}{c}{\(\eta_{h}\)} \\
      \cmidrule(rl){2-5}
    & \multicolumn{2}{c}{Primal}      & \multicolumn{2}{c}{Mixed} \\
        \cmidrule(rl){2-3}                \cmidrule(rl){4-5}
\(k\) & Uniform & Box    & Uniform & Box  \\
\midrule
 16  & 0.2264 & 0.2269 & 0.2830 & 0.2806 \\
 32  & 0.1410 & 0.1411 & 0.2386 & 0.2394 \\
 64  & 0.1225 & 0.1226 & 0.2083 & 0.2088 \\
128  & 0.1130 & 0.1130 & 0.1796 & 0.1798 \\
\bottomrule
\end{tabular}
}
\caption{Convergence rate \(\eta_{h}\) of the outer FGMRES iterations for the primal and mixed formulations with each source type when \(B_{\hat{\delta}}\) and \(A_{\hat{\delta}}\) are approximated with \(k\) HSS iterations and each HSS iteration is approximated with one or two W cycles for the primal or mixed formulations respectively.}
\label{tab:shift-mg-rate}
\end{table}

\subsection{Parallel performance}

Finally, we examine the algorithm's behaviour in a high performance
computing setting, focussing on the primal formation, with larger $k$,
and kept $\theta=1$. The computations were done on ARCHER2, the UK national
supercomputer \cite{Beckett2024}. ARCHER2 has 128 cores per node, however we
used 64 cores per node in this computation since we are otherwise limited by
memory bandwidth due to the low arithmetic intensity of Jacobi smoothers.

We use a brief performance model to judge the parallel scaling.
The total number of DoFs in \(d\) dimensions scales as \(DoF=\mathcal{O}((k^{3/2})^d)\).
For standard multigrid components, the communication load per processor is independent of the wavenumber so we can use $P=\mathcal{O}(k^{3d/2})$ processors and expect very good scaling from the multigrid algorithm.
The work per HSS iteration scales with the DoF count because the number of multigrid cycles and levels are fixed, so if the number of HSS iterations is $\mathcal{O}(k)$ then the total work is $W=\mathcal{O}(k^{1+(3d/2)})$.
Therefore, we expect the wallclock time to scale linearly with the wavenumber $T=W/P=\mathcal{O}(k)$.
In 2 dimensions this results in \(\mathcal{O}(k^{3})\) DoFs and processors and \(\mathcal{O}(k^{4})\) computational work.

The convergence rates are shown in Table \ref{tab:weak-scaling-2D-convergence}; we see that our observations in previous experiments are extended to higher $k$ values, up to $k=1024$ in the largest case.
Once again the HSS convergence bounds are tight, and we see further evidence of the convergence rate for the outer iterations approaching a stable value for larger $k$.
While the results using fixed point Richardson iterations verify the theoretical results, we also show that, in practice, FGMRES can be used to significantly reduce the total number of iterations required - in these examples by more than half.
In Table \ref{tab:weak-scaling-2D-timing} and in Figure \ref{fig:timing-2D} we show parallel timing results using \(\mathcal{O}(k^{3})\) processors.
We see very robust linear wallclock time scaling with $k$, as we hoped, for both examples and with either Richardson or FGMRES outer iterations.

The theoretical results in sections \ref{sec:HSS shifted} and \ref{sec:HSS 0} did not rely on any assumption on the number of spatial dimensions of the problem.
To verify that the results are in fact robust to the spatial dimension we show some results for the 3D Helmholtz equation.
Tables \ref{tab:weak-scaling-3D-convergence} and \ref{tab:weak-scaling-3D-timing} show convergence and timing results respectively for the uniform source and the (3D extension of) box source examples.
In 3D the number of DoFs scales with \(\mathcal{O}(k^{4.5})\) so we show a narrower wavenumber range \(k\in\{20,40,80\}\), but we still see very similar behaviour to the 2D examples.
In Table \ref{tab:weak-scaling-3D-convergence} we see that the iteration counts are almost constant over this wavenumber range, and that the convergence bound for the HSS iterations for the shifted operator are again very tight.
The timing results in Table \ref{tab:weak-scaling-3D-convergence} and Figure \ref{fig:timing-3D}
show that we again achieve wallclock times which scale linearly with \(k\).

\begin{figure}
\centering
\begin{subfigure}{0.49\textwidth}
   \centering
   \includegraphics[width=0.99\textwidth]{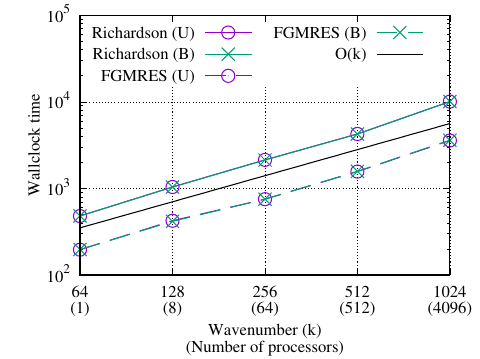}
   \caption{}
   \label{fig:timing-2D}
\end{subfigure}
\begin{subfigure}{0.49\textwidth}
   \centering
   \includegraphics[width=0.99\textwidth]{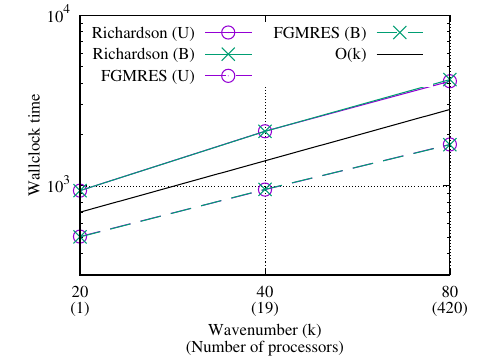}
   \caption{}
   \label{fig:timing-3D}
\end{subfigure}
\caption{Weak scaling for the primal formulation in 2D (\ref{fig:timing-2D}) and 3D (\ref{fig:timing-3D}) 3D.}
\label{fig:scaling}
\end{figure}

\begin{table}
  \centering
  \begin{tabular}{rccccc}
    \multicolumn{6}{c}{Richardson}\\
    \toprule
  \(k\) & {\(M_{h}\)} & \(\eta_{h}\) & {\(\eta_{s}\)} & {\(\nu_{s}\)} & \(\eta_{mg}\) \\
    \midrule
  64 & 18 & 0.4903 & 0.9659 & 0.9692 & 0.00409 \\
 128 & 18 & 0.4899 & 0.9833 & 0.9845 & 0.00284 \\
 256 & 18 & 0.4844 & 0.9917 & 0.9922 & 0.00165 \\
 512 & 17 & 0.4720 & 0.9960 & 0.9961 & 0.00086 \\
1024 & 18 & 0.4957 & 0.9978 & 0.9980 & 0.00057 \\
\bottomrule
  \end{tabular}

    \begin{tabular}{rccccc}
    \multicolumn{6}{c}{FGMRES}\\
    \toprule
  \(k\) & {\(M_{h}\)} & \(\eta_{h}\) & {\(\eta_{s}\)} & {\(\nu_{s}\)} & \(\eta_{mg}\) \\
  \midrule
  64 & 8 & 0.1760 & 0.9677 & 0.9692 & 0.00336 \\
 128 & 8 & 0.1597 & 0.9837 & 0.9845 & 0.00199 \\
 256 & 7 & 0.1316 & 0.9920 & 0.9922 & 0.00109 \\
 512 & 7 & 0.1247 & 0.9960 & 0.9961 & 0.00072 \\
1024 & 7 & 0.1222 & 0.9980 & 0.9980 & 0.00057 \\
\bottomrule
\end{tabular}
  \caption{Weak scaling convergence for 2D Helmholtz with a box source. Iteration counts \(M_{h}\) and contraction rates \(\eta_{h}\) for the Helmholtz operator \eqref{eq:primal} without shift; HSS iteration contraction rates, measured rate \(\eta_{s}\) and expected rate \(\nu_{s}\) from Eq. \(\eqref{eq:primal-hss-contraction}\) for the Helmholtz operator \eqref{eq:primal} with shift \(\hat{\delta}\); multigrid contraction rates \(\eta_{mg}\) for the shifted Helmholtz operator \eqref{eq:primal-hss-system}. With a uniform source the iteration counts \(M_{h}\) are identical and the contraction rates \(\eta\) are almost identical.}
  \label{tab:weak-scaling-2D-convergence}
\end{table}

\begin{table}
  \centering
\setlength{\tabcolsep}{5pt}
\begin{tabular}{rrcrr}
   \multicolumn{5}{c}{Richardson}\\
   \toprule
   \(k\) & {\(P\)} & \(DoF\) & {\(T\) (U)} & {\(T\) (B)} \\
   \midrule
  64 &    1 & \(6.60\!\times\!10^4\) &   479 &   481 \\
 128 &    8 & \(5.20\!\times\!10^5\) &  1039 &  1034 \\
 256 &   64 & \(4.20\!\times\!10^6\) &  2132 &  2123 \\
 512 &  512 & \(3.37\!\times\!10^7\) &  4248 &  4234 \\
1024 & 4096 & \(2.68\!\times\!10^8\) & 10054 & 10042 \\
\bottomrule
\end{tabular}\\
\begin{tabular}{rrcrr}
   \multicolumn{5}{c}{FGMRES}\\
   \toprule
   \(k\) & {\(P\)} & \(DoF\) & {\(T\) (U)} & {\(T\) (B)} \\
   \midrule
  64 &    1 & \(6.60\!\times\!10^4\) &  196 &   197 \\
 128 &    8 & \(5.20\!\times\!10^5\) &  423 &   416 \\
 256 &   64 & \(4.20\!\times\!10^6\) &  751 &   750 \\
 512 &  512 & \(3.37\!\times\!10^7\) & 1564 &  1573 \\
1024 & 4096 & \(2.68\!\times\!10^8\) & 3568 &  3599 \\
\bottomrule
\end{tabular}
  \caption{Weak scaling timing for the 2D Helmholtz equation with wavenumber \(k\) with Richardson and FGMRES outer iterations. Number of processors \(P\); degrees of freedom \(DoF\); wallclock time \(T\) with uniform (U) and box (B) source terms. \(66\times10^{3}\) degrees of freedeom per processor. Wallclock time is expected to scale with \(\mathcal{O}(k)\).}
  \label{tab:weak-scaling-2D-timing}
\end{table}

\begin{table}
  \centering
  \begin{tabular}{rccccc}
    \multicolumn{6}{c}{Richardson}\\
    \toprule
  \(k\) & {\(M_{h}\)} & \(\eta_{h}\) & {\(\eta_{s}\)} & {\(\nu_{s}\)} & \(\eta_{mg}\) \\
    \midrule
20  & 17 & 0.4539 & 0.9001 & 0.9048 & 0.02034 \\
40  & 18 & 0.4848 & 0.9466 & 0.9512 & 0.01794 \\
80  & 17 & 0.4612 & 0.9746 & 0.9753 & 0.01125 \\
\bottomrule
  \end{tabular}

    \begin{tabular}{rccccc}
    \multicolumn{6}{c}{FGMRES}\\
    \toprule
  \(k\) & {\(M_{h}\)} & \(\eta_{h}\) & {\(\eta_{s}\)} & {\(\nu_{s}\)} & \(\eta_{mg}\) \\
  \midrule
20  & 10 & 0.2587 & 0.9029 & 0.9048 & 0.01931 \\
40  &  9 & 0.2032 & 0.9492 & 0.9512 & 0.01507 \\
80  &  8 & 0.1677 & 0.9748 & 0.9753 & 0.00950 \\
\bottomrule
\end{tabular}
  \caption{Weak scaling convergence for 3D Helmholtz with a box source. Iteration counts \(M_{h}\) and contraction rates \(\eta_{h}\) for the Helmholtz operator \eqref{eq:primal} without shift; HSS iteration contraction rates, measured rate \(\eta_{s}\) and expected rate \(\nu_{s}\) from Eq. \(\eqref{eq:primal-hss-contraction}\) for the Helmholtz operator \eqref{eq:primal} with shift \(\hat{\delta}\); multigrid contraction rates \(\eta_{mg}\) for the shifted Helmholtz operator \eqref{eq:primal-hss-system}. With a uniform source the iteration counts \(M_{h}\) are identical and the contraction rates \(\eta\) are almost identical.}
  \label{tab:weak-scaling-3D-convergence}
\end{table}

\begin{table}
  \centering
\setlength{\tabcolsep}{5pt}
\begin{tabular}{rrcrr}
   \multicolumn{5}{c}{Richardson}\\
   \toprule
   \(k\) & {\(P\)} & \(DoF\) & {\(T\) (U)} & {\(T\) (B)} \\
   \midrule
20  &   1 & \(1.18\!\times\!10^5\) &  937 &  936 \\
40  &  19 & \(2.15\!\times\!10^6\) & 2087 & 2083 \\
80  & 420 & \(4.70\!\times\!10^7\) & 4097 & 4189 \\
\bottomrule
\end{tabular}\\
\begin{tabular}{rrcrr}
   \multicolumn{5}{c}{FGMRES}\\
   \toprule
   \(k\) & {\(P\)} & \(DoF\) & {\(T\) (U)} & {\(T\) (B)} \\
   \midrule
20  &   1 & \(1.18\!\times\!10^5\) &  504 &  502 \\
40  &  19 & \(2.15\!\times\!10^6\) &  950 &  952 \\
80  & 420 & \(4.70\!\times\!10^7\) & 1745 & 1741 \\
\bottomrule
\end{tabular}\\
  \caption{Weak scaling timing for the 3D Helmholtz equation with wavenumber \(k\) with Richardson and FGMRES outer iterations. Number of processors \(P\); degrees of freedom \(DoF\); wallclock time \(T\) with uniform (U) and box (B) source terms. \(118\times10^{3}\) degrees of freedeom per processor. Wallclock time is expected to scale with \(\mathcal{O}(k)\).}
  \label{tab:weak-scaling-3D-timing}
\end{table}

\section{Summary and outlook}
\label{sec:summary}

In this article we introduced and analysed a preconditioner for the
indefinite Helmholtz equation based on HSS iteration applied to a
$\mathcal{O}(1)$ shifted operator. The HSS iteration requires the
solution of another $\mathcal{O}(k)$ shifted operator which is known
to be suitable for multigrid with standard components. We proved and
demonstrated numerically that when $\mathcal{O}(k)$ HSS inner iterations
are used in the preconditioner, $\mathcal{O}(1)$ outer iterations are
required, i.e. the number of outer iterations is independent of $k$
and the mesh resolution. We then demonstrated numerically that when
the inverse $\mathcal{O}(k)$ shifted operator is replaced by a
multigrid cycle using standard components, the robustness still holds.
Thus, we can rely upon standard multigrid parallel scalability for the
solver, with the only serial component being the $k$ inner iterations.
Therefore we expect $\mathcal{O}(k)$ wallclock times within the range
of optimal scalability of the multigrid for a given mesh resolution,
which we demonstrated in practice up to high wavenumbers.

In future work we will extend our analysis and computational experiments
to higher-order discretisations, H(div)-$L_2$ mixed formulations, as
well as harmonic elastodynamics and harmonic Maxwell equations.

\section*{Acknowledgements}
This work used the ARCHER2 UK National Supercomputing Service
(https://www.archer2.ac.uk). The authors are grateful to Niall
Bootland, Martin Gander, Ivan Graham, and Euan Spence, for their
useful discussions and feedback about this work.

\bibliographystyle{siamplain}
\bibliography{helm}
\end{document}